\documentclass{article}

\usepackage{arxiv}

\usepackage[utf8]{inputenc} 
\usepackage[T1]{fontenc}    
\usepackage{hyperref}       
\usepackage{url}            
\usepackage{booktabs}       
\usepackage{amsfonts}       
\usepackage{nicefrac}       
\usepackage{microtype}      
\usepackage{lipsum}
\usepackage{graphicx}
\usepackage{epstopdf}
\RequirePackage[OT1]{fontenc}
\RequirePackage{amsthm,amsmath, amsfonts, mathrsfs}
\RequirePackage[round, authoryear]{natbib}
\RequirePackage{graphicx, multirow, booktabs, subfig}
\usepackage[utf8]{inputenc}
\RequirePackage{lipsum}
\RequirePackage{bm}
\usepackage{caption}

\title{On minimum Bregman divergence inference.}

\author{
 Soumik Purkayastha \\
  Department of Biostatistics\\
  University of Michigan\\
  Ann Arbor, MI, USA \\
  \texttt{soumikp@umich.edu} \\
   \And
 Ayanendranath Basu \\
 Interdisciplinary Statistical Research Unit\\
 Indian Statistical Institute\\
 Kolkata, WB, INDIA \\
 \texttt{ayanbasu@isical.ac.in} \\
}

\numberwithin{equation}{section}
\theoremstyle{plain}
\newtheorem{remark}{Remark}
\newtheorem{thm}{Theorem}[section]

\begin{document}
\maketitle
\begin{abstract}
In this paper a new family of minimum divergence estimators based on the Bregman divergence is proposed.
The popular density power divergence (DPD) class of estimators is a sub-class of Bregman divergences.
We propose and study a new sub-class of Bregman divergences called the exponentially weighted divergence (EWD).
Like the minimum DPD estimator, the minimum EWD estimator is recognised as an M-estimator. This characterisation is useful while discussing the asymptotic behaviour as well as  the robustness properties of this class of estimators. Performances of the two classes are compared -- both through simulations as well as through real life examples. We develop an estimation process not only for independent and homogeneous data, but also for non-homogeneous data. 
General tests of parametric hypotheses based on the Bregman divergences are also considered. We establish the asymptotic null distribution of our proposed test statistic and explore its behaviour when applied to real data. The inference procedures generated by the new EWD divergence appear to be competitive or better that than the DPD based procedures.
\end{abstract}


\section{Introduction}
\label{sec:intro}  

Density based minimum divergence methods are popular tools in statistical inference. In case of the estimation problem, this amounts to estimating the parameters of interest
by minimising an empirical version of some suitably chosen divergence between the `true' density underlying the data and the assumed model density. Many of these methods combine strong robustness properties with high asymptotic efficiency, which is one of the reasons for their popularity.
An important class of density based divergences useful in this context is the class of $\phi$ divergences (see \cite{csiszar1963}). Under standard regularity conditions, all minimum $\phi$ divergence estimators have full asymptotic efficiency at the model (\cite{lindsay94}); many of them also have attractive robustness properties. A seminal work by \cite{beran1977}, who investigated the minimum Hellinger distance estimator (MHDE), appears to be the first which demonstrated that strong robustness properties may be achieved simultaneously with full asymptotic efficiency. Later, the same has been demonstrated with respect to much of the $\phi$ divergence class (see, eg., \cite{basu2011}). The usefulness of the corresponding procedures in providing robust alternatives to the likelihood ratio test has also been explored in the literature (\cite{simpson1989}; \cite{lindsay94}; \cite{basu2011}). The extension of this approach to problems beyond the simple i.i.d. set-up
has also been attempted by several later authors. On the whole, the utility of the minimum divergence procedures based on $\phi$ divergences is well established in the literature.

One of the major criticisms
 of this inferential procedure is that it involves the use of some form of non-parametric smoothing (such as kernel-based density estimation) to produce a continuous estimate of the true density (which is necessary to construct the divergence when the model density is continuous). While kernel density estimation (or other suitable non-parametric smoothing techniques) represents a very important class of statistical procedures, it involves a lot of complication 
 and the bandwidth selection issue can throw up many potential difficulties.
 The slow convergence of the kernel density estimator to the `truth' for high dimensional data adds another facet to the problem. The complicated nature of the estimating equation also makes the theoretical derivations harder. Development of methods which eliminate these difficulties may be worthwhile even if they involve a small loss in asymptotic efficiency.

An alternative class of minimum divergence estimators which does not require non-parametric smoothing in the construction of the empirical divergence is the class of minimum Bregman divergence estimators.
An important example of divergences in this class is the family of  density power divergences (DPD($\alpha$), where $\alpha$ is the tuning parameter); the corresponding minimum density power divergence estimators (MDPDE($\alpha$)) have been shown to combine strong robustness properties with high asymptotic efficiency (see \cite{basu1998}). Divergences within the Bregman class have been called decomposable divergences by \cite{broniatowski2012decomposable} and non-kernel divergences by \cite{jana2019characterization}. These divergences have simple estimating equations and much of their asymptotic properties can be obtained from the M-estimation theory. The  Kullback-Leibler divergence, which is a decomposable divergence, is the only common member between the $\phi$ divergence class and the Bregman divergence class.

In the context of robust parametric estimation, specifically in the context of density-based minimum divergence estimation with a view to robustness, we have several `good' choices available. In order to justify the development of another family of estimators, one must demonstrate that the new estimators are competitive, if not better than the existing standard. Within the class of  minimum divergence estimators which do not require any nonparametric smoothing, the MDPDE($\alpha$) is the current standard. In this paper, we  will develop a family of divergences yielding minimum divergence estimators which appear to satisfy this requirement; at the least, this family provides a highly competitive standard.
Our proposed class of divergences will be called the exponentially weighted divergence family, indexed by a tuning parameter $\beta$ (henceforth referred to as EWD($\beta$)). The corresponding minimum exponentially weighted divergence estimator will be denoted by MEWDE($\beta$).

Like estimation, hypothesis testing is another fundamental area of statistical inference. Although the likelihood ratio test is a well studied component of classical hypothesis testing theory, it is known to have very poor robustness under model misspecification and presence of outliers. Many density based minimum distance procedures have been observed to have strong robustness properties in estimation and testing together with high efficiency, eg.,  \cite{pardo2006statistical} and \cite{basu2013testing, basu2018testing}. The last mentioned paper addresses the general problem of parametric hypothesis testing of composite null hypotheses based on the density power divergence alone. The theoretical results presented in the present paper extend the said testing procedure to the entire class of Bregman divergence, with special emphasis on our proposed EWD($\beta$) class of divergences.

\section{The exponentially weighted divergence}
\label{sec: ewd}
Originally defined in the context of convex programming by \cite{bregman1967}, the {\em Bregman divergence} for $p,q \in \mathbb{R}^d$ is defined as
\begin{equation*}
{D_{B}(p,q)=B(p)-B(q)-\langle \nabla B(q),p-q\rangle},
\end{equation*}
where $\nabla$ denotes the gradient of a function with respect to its arguments and $\langle x, y\rangle$ denotes the inner product of $x$ and $y$. The function $B: \mathbb{R}^d \rightarrow \mathbb{R}$ is strictly convex and consequently, the measure $D_B(x, y)$ is non-negative and equals zero only when $x = y$. Extending this formulation to the case of two probability density functions (pdf) ${g}$ and ${f}$, we define the divergence
\begin{equation}
{D_{B}(g, f)=\int_x [B(g(x))-B(f(x))- (g(x)-f(x))B'(f(x))]dx},
\label{eq1}
\end{equation}
where $B'(\cdot)$ is the derivative of $B$ with respect to its argument. Since the integrand is non-negative for each $x$, it follows that $D_B(g, f)$ is non-negative. Moreover, the measure is zero when its arguments are identically equal. \cite{csiszar1991} discuss this and similar measures in greater detail. We note that the convex functions $B(y)$ and $B^{*}(y) = B(y) + ay + c$ generate identical divergences in Equation (\ref{eq1}) for $a, c \in \mathbb{R}$.

The minimum Bregman divergence estimation procedure based on a general convex $B$ function may be described as follows. Given an i.i.d. random sample $X_1, \ldots, X_n$ from the distribution $G$, we model these data by a parametric family (of densities $f_{\theta}$, indexed by the parameter $\theta$) $\mathscr{F}_{\bm{\theta}} := \{f_{\theta}:\ \theta \in \Omega \subset {\mathbb R}^p\}$. 
Specifically, we wish to estimate the value of the model parameter $\bm{\theta}$ by choosing the model density which gives the closest fit to the data 
in the minimum Bregman divergence sense. Let $g$ and $f_\theta$ be the density functions associated with distribution functions $G$ and $F_\theta$ respectively. 
An empirical version  $D_B(g, f)$, given by the right side of Equation (\ref{eq1}) with $f$ replaced by the model element $f_\theta$  may now be obtained as 
\begin{equation*}
  \int   [ B(g(x)) - B(f_{\bm{{\bm{\theta}}}}(x))] dx - n^{-1} \sum_{i=1}^{n} B'(f_{\bm{{\bm{\theta}}}}(X_i)) + \int B'(f_{\bm{{\bm{\theta}}}}(x)) f_{\bm{{\bm{\theta}}}}(x) dx.
\end{equation*}
Here we have replaced the theoretical mean $\int B'(f_{\bm{{\bm{\theta}}}}(x)) g(x)  dx$ with the sample mean based on $X_1, \ldots, X_n$. Grouping the terms of the above equation including terms based only on $f_{\bm{{\bm{\theta}}}}$, terms based both on $g$ and $f_{\bm{{\bm{\theta}}}}$ and terms based only on $g$, the above objective function may be expressed as
\begin{equation}
\label{eq2}
  \int  \big[f_{{\bm{{\bm{\theta}}}}}(x)B'(f_{{\bm{{\bm{\theta}}}}}(x)) - B(f_{{\bm{{\bm{\theta}}}}}(x))\big] dx - n^{-1} \sum_{i=1}^{n} B'(f_{{\bm{{\bm{\theta}}}}}(X_i)) + \int   B(g(x)) dx.
\end{equation}
The last term in the above expression may be ignored as it has no role in optimisation over ${\bm{{\bm{\theta}}}} \in \Omega$.

Let ${u_{{\bm{{\bm{\theta}}}}}}(x) = \nabla_{{\bm{{\bm{\theta}}}}} \log(f_{{\bm{{\bm{\theta}}}}}(x))$ be the likelihood score function of the model being considered, where $\nabla_{\bm{{\bm{\theta}}}}$ represents the gradient of a function with respect to ${\bm{{\bm{\theta}}}}$. Under differentiability of the model with respect to ${\bm{{\bm{\theta}}}}$, minimisation of the expression in Equation (\ref{eq2}) leads to the estimating equation
\begin{equation}
\label{eq3}
    n^{-1} \sum_{i=1}^{n} u_{{\bm{{\bm{\theta}}}}}(X_i) B''(f_{{\bm{{\bm{\theta}}}}}(X_i))f_{{\bm{{\bm{\theta}}}}}(X_i) = \int u_{{\bm{{\bm{\theta}}}}}(x)  B''({f_{\bm{{\bm{\theta}}}}}(x)) f_{{\bm{{\bm{\theta}}}}}^2(x) dx, 
\end{equation}
where $B''(\cdot) (\text{or } B'''(\cdot))$ is the indicated second (or third) derivative. This may be viewed as a generalised likelihood equation, or a weighted likelihood equation having the form
\begin{equation}
\label{eq3a}
    n^{-1} \sum_{i=1}^{n}  u_{{\bm{{\bm{\theta}}}}}(X_i) w(f_{{\bm{{\bm{\theta}}}}}(X_i)) = \int u_{{\bm{{\bm{\theta}}}}}(x) w(f_{{\bm{{\bm{\theta}}}}}(x))  f_{{\bm{{\bm{\theta}}}}}(x)    dx,
\end{equation}
where the usual likelihood score equation is recovered for $w(t) \equiv 1 $. Comparing Equations (\ref{eq3}) and (\ref{eq3a}), we obtain
\begin{equation}
\label{eq4}
    w(f_{{\bm{{\bm{\theta}}}}}(x))=B''({f_{\bm{{\bm{\theta}}}}}(x)) f_{{\bm{{\bm{\theta}}}}}(x).
\end{equation}
Using convexity of $B$ and non negativity of $f_{\bm{{\bm{\theta}}}}$, it follows that $w$ will be non-negative. We see that the Kullback-Leibler divergence, the  $L_2$ divergence and more generally, the density power divergence DPD($\alpha$)  are all special cases of the Bregman divergence; the corresponding $B$ functions are $x\log(x)$, $x^2$ and $\alpha^{-1}x^{1+\alpha}$ respectively. The Kullback-Leibler divergence is
\begin{equation*}
  D_{KL}(g, f_{{\bm{{\bm{\theta}}}}}) = \int g(x)\ \log \Bigg(\frac{g(x)}{f_{{\bm{{\bm{\theta}}}}}(x)}\Bigg) dx,
\end{equation*}
while the (squared) $L_2$ distance is
\begin{equation*}
  D_{L_{2}}(g, f_{{\bm{{\bm{\theta}}}}}) = \int [g(x) - f_{{\bm{{\bm{\theta}}}}}(x)]^2dx,
\end{equation*}
and the general form of DPD($\alpha$) is
\begin{equation*}
  DPD_{\alpha}(g, f_{{\bm{{\bm{\theta}}}}}) = \int \Big[ f_{{\bm{{\bm{\theta}}}}}^{1+\alpha}(x) - \Big(1+\frac{1}{\alpha}\Big) g(x) f^{\alpha}_{{\bm{{\bm{\theta}}}}}(x) + \frac{1}{\alpha} g^{1+\alpha}(x)\Big] dx,\ \ \  \alpha > 0.
\end{equation*}

In order to develop new estimation procedures based on Bregman divergences, one can do one of two things: (a) start with a specific convex function $B$ and construct a weighted likelihood equation as given in Equation (\ref{eq3}), or (b) begin with a suitable weight function (motivated by considerations of robustness), and the associated weighted likelihood representation as in Equation (\ref{eq3a}) and backtrack to recover the corresponding convex function $B$. We choose the latter approach. See \cite{biswas2019} for a general discussion on Bregman divergences and weighted likelihood.

Philosophically, our treatment of outliers is probabilistic, in that an outlying point is one which has a small probability of occurrence under a given model $f_{{\bm{{\bm{\theta}}}}} \in \mathscr{F}_{{\bm{{\bm{\theta}}}}}.$ We choose to downweight those observations in the estimating equation for which the value of $f_{{\bm{\theta}}}(x)$ is small. We plot the weight functions defined in Equation (\ref{eq5}) for some members of the DPD($\alpha$) family at different values of $\alpha$ in Figure \ref{fig:f1}. While the weight function is a constant (equal to 1) at $\alpha=0$, all positive values of $\alpha$ downweight observations  having density $f_{\bm{\theta}}(x) \leq 1.$ The strength of downweighting increases with increasing $\alpha$.  For $f_{\bm{\theta}}(x) > 1$, the weights grow unboundedly for all $\alpha > 0$ as the argument increases. The measure DPD(0) corresponds, in a limiting sense, to the Kullback-Leibler divergence which is minimized by the MLE. From Figure \ref{fig:f1}, we note that the MLE gives equal weight to all observations, including outlying ones, leading to its poor robustness properties. The measure DPD(1) corresponds to the squared $L_2$ distance.
\begin{figure}[ht]
  \centering
  \includegraphics[width = 0.8\textwidth]{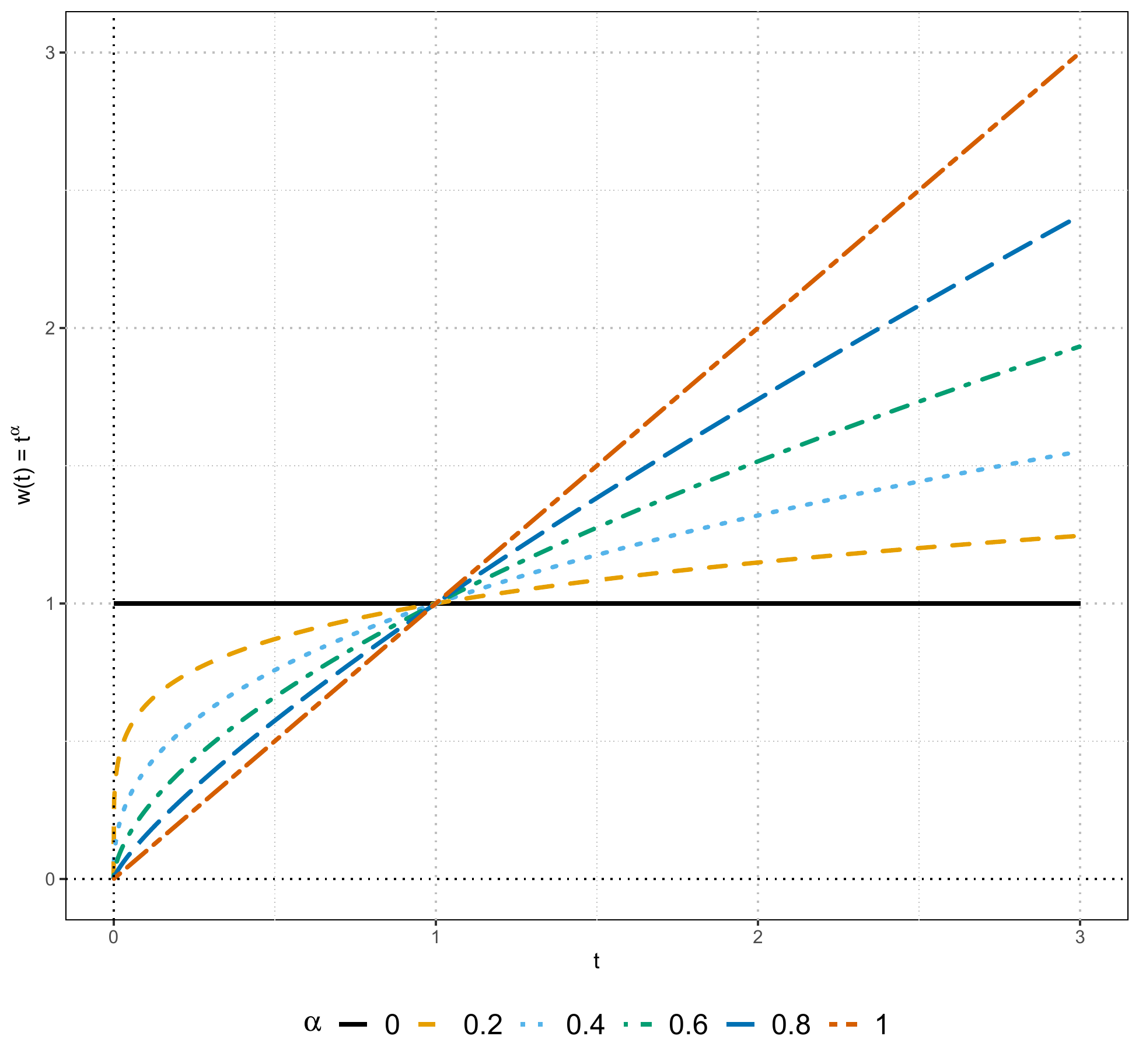}
  \caption{Weight functions of some DPD($\alpha$) members. }\label{fig:f1}
\end{figure}

We propose a new class of divergences based on a different choice of the weight function
\begin{equation}
\label{eq5}
    w_{\beta}(t) = \left\{
	\begin{array}{ll}
		1 - \exp(-t/\beta) & \mbox{if } \beta > 0, \\
		1 & \mbox{if } \beta = 0.
	\end{array}
\right.
\end{equation}
%
%
%
%
%
These weights smoothly drop to zero for decreasing values of the probability density function $f_{\bm{\theta}}(x)$ for $\beta > 0$. However, unlike the DPD($\alpha$) weights, they are bounded above by 1. We plot the weight functions given by Equation (\ref{eq5}) for specific values of $\beta$ in Figure \ref{fig:f2}.
\begin{figure}[ht]
  \centering
  \includegraphics[width = 0.8\textwidth]{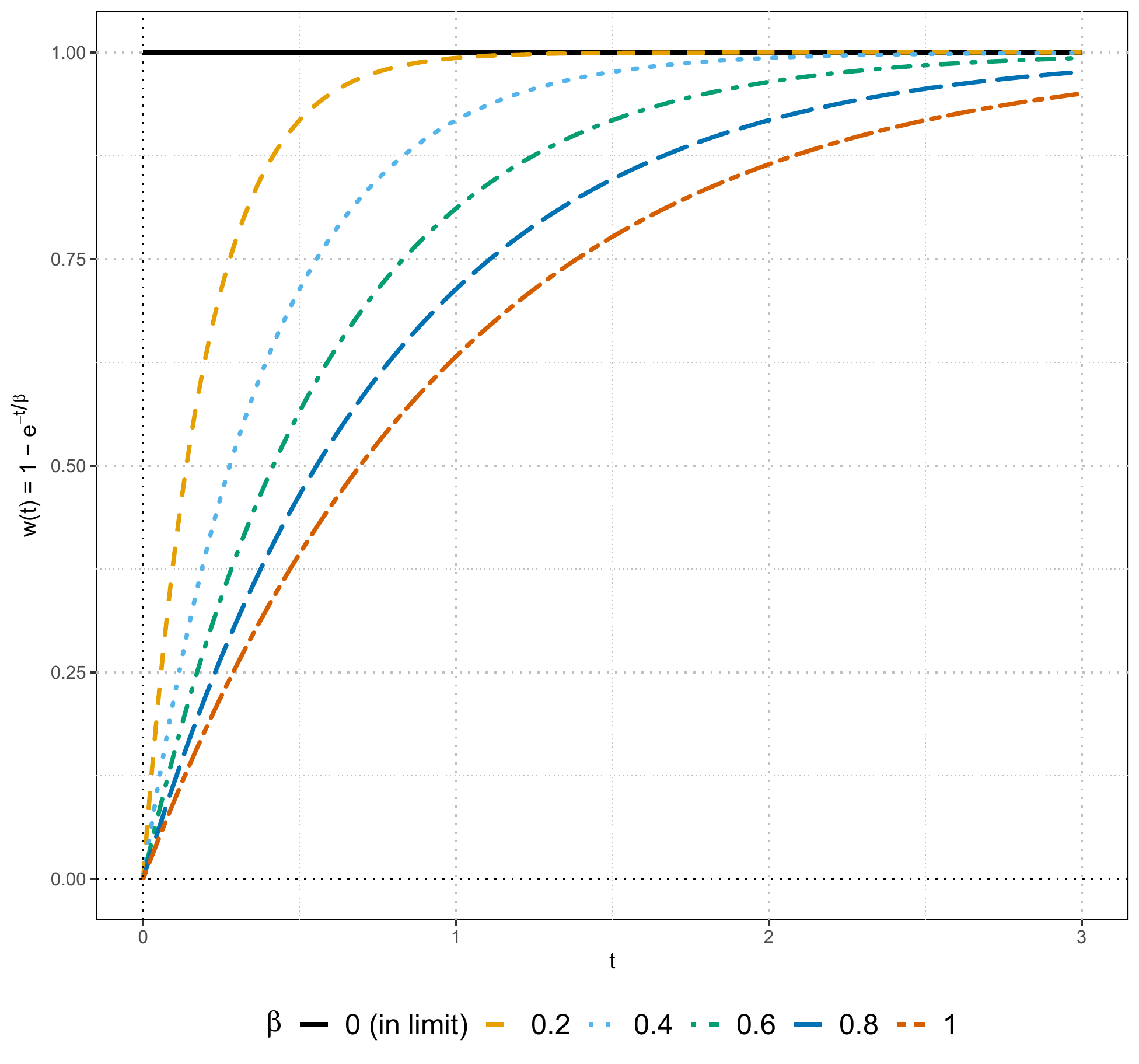}
  \caption{Weight functions of some EWD($\beta$) members. }\label{fig:f2}
\end{figure}
The likelihood equation may be recovered at $\beta = 0$, where, to avoid the complications of division by zero, the weights have been defined by the corresponding limiting case as $\beta \rightarrow 0$. Using Equation (\ref{eq4}), we recover the divergence (or rather, the associated $B$ function). This function is given by
\begin{equation*}
    B(x) = \frac{x^2}{\beta}\Big[ \sum_{n=0}^{\infty} \frac{(-x/\beta)^n}{(n+2)! (n+1) }\Big].
\end{equation*}
In Appendix \ref{appA}, we show that this can be further simplified to
\begin{equation}
\label{eq6}
  B(x) = -x + \gamma x + \beta  - \beta \exp(-x/\beta) + x \Gamma(0, x/\beta) + x\log(x/\beta),
\end{equation}
where $\gamma$ is the Euler-Mascheroni constant $$\gamma = \underset{n \rightarrow \infty}{\lim} \Bigg(\sum_{k=1}^{n} \frac{1}{k} - \log n\Bigg) = - \int_{0}^{\infty} \log(t)\exp(-t) dt,$$ and $\Gamma(\alpha,\beta)$ is the incomplete Gamma integral defined as $$\Gamma(\alpha,\beta) = \int_{\beta}^{\infty} y^{\alpha-1} \exp(-y) dy.$$  The associated Bregman divergence (which we will refer to as the {exponentially weighted divergence} EWD($\beta)$) has the form
\begin{equation}
\label{eq.ewd}
    \text{EWD}(\beta) = \int  \big[f_{{\bm{\theta}}}(x)B'(f_{{\bm{\theta}}}(x)) - B(f_{{\bm{\theta}}}(x))\big] dx - n^{-1} \sum_{i=1}^{n} B'(f_{{\bm{\theta}}}(X_i)),
\end{equation}
where $B$ is given by Equation (\ref{eq6}). The resultant estimating equation has the form
\begin{equation*}
     n^{-1} \sum_{i=1}^{n} u_{{\bm{\theta}}}(X_i) [1-\exp(-f_{{\bm{\theta}}}(X_i)/\beta)] = \int u_{{\bm{\theta}}}(x) [1-\exp(-f_{{\bm{\theta}}}(x)/\beta)] f_{{\bm{\theta}}}(x) dx.
\end{equation*}
We note that the EWD($\beta$) family can also be generated by the simplified $B$ function
\begin{equation*}
B(x) = - \beta \exp(-x/\beta) + x \Gamma(0, x/\beta) + x\log(x/\beta).
\end{equation*}

\section{Properties}
\subsection{Link with M-estimation}
\label{subsec:m-est}
Given i.i.d. observations $X_1, \ldots, X_n$ from a distribution modeled by the parametric family $\mathscr{F}_{\bm{{\bm{\theta}}}}$, an M-estimator of the target parameter ${\bm{{\bm{\theta}}}}$ may be obtained by solving an estimating equation of the form $\sum_{i=1}^{n} \psi (X_i, {\bm{{\bm{\theta}}}}) = 0$ (see, e.g., \cite{huber2009} and \cite{hampel1986} for more details). Any minimum Bregman divergence estimator is also an M estimator. The $\psi$ function associated with the minimum Bregman divergence estimator is
\begin{equation}
\label{eq7}
    \psi(x, {\bm{{\bm{\theta}}}}) = u_{{\bm{{\bm{\theta}}}}}(x) B''(f_{{\bm{{\bm{\theta}}}}}(x))f_{{\bm{{\bm{\theta}}}}}(x) - \int_t u_{{\bm{{\bm{\theta}}}}}(t) B''(f_{{\bm{{\bm{\theta}}}}}(t))f^2_{{\bm{{\bm{\theta}}}}}(t) dt.
\end{equation}
We note that the $\psi$ function in this case makes explicit use of the form of the pdf of the model unlike the location-scale form used commonly in M-estimation. For MEWDE($\beta$), in particular, the associated $\psi$ function is
\begin{equation*}
    \psi(x, {\bm{{\bm{\theta}}}}) = u_{{\bm{{\bm{\theta}}}}}(x) [1 - \exp(-f_{{\bm{{\bm{\theta}}}}}(x)/\beta)] - \int_t u_{{\bm{{\bm{\theta}}}}}(t) [1 - \exp(-f_{{\bm{{\bm{\theta}}}}}(t)/\beta)]f_{{\bm{{\bm{\theta}}}}}(t) dt.
\end{equation*}

\subsection{Asymptotic properties}
\label{subsec:asymp}
We present results related to the asymptotic distribution of any minimum Bregman divergence based estimator in general and the MEWDE($\beta$) in particular, when the true distribution $G$ from which the data are generated is not necessarily in the model under study. The theoretical estimating equation is $\int \psi(x, {\bm{{\bm{\theta}}}})dG(x) = 0$, where $\psi(x, {\bm{{\bm{\theta}}}})$ is given by Equation (\ref{eq7}). We observe that the functional $T_{\beta}(G)$ associated with MEWDE($\beta$) is  Fisher consistent; it recovers the value ${\bm{{\bm{\theta}}}}_0$ when the true distribution $G = F_{{\bm{{\bm{\theta}}}}_0}$ is a member of the parametric family being used to model the given data (this is true for any minimum Bregman divergence based estimator in general). When $G$ is not in the model, our best fitting parameter ${\bm{{\bm{\theta}}}}_g = T_{\beta}(G)$ will be the root of the theoretical estimating equation
\begin{equation}
\label{eq8}
  \int_{x} u_{{\bm{{\bm{\theta}}}}}(x)B''(f_{{\bm{{\bm{\theta}}}}}(x))f_{{\bm{{\bm{\theta}}}}}(x) dG(x) = \int_{x} u_{{\bm{{\bm{\theta}}}}}(x)B''(f_{{\bm{{\bm{\theta}}}}}(x))f_{{\bm{{\bm{\theta}}}}}(x)dF_{\bm{{\bm{\theta}}}}(x).
\end{equation}
Let $X_1, \ldots, X_n$ be a random sample from the distribution $G$ having density $g$. The minimum Bregman divergence estimator for this sample is obtained as a solution of Equation (\ref{eq3}) via the minimization of the quantity given by Equation (\ref{eq2}) for a given $B$ function. We define
\begin{equation}
\label{eq9}
\begin{aligned}
    H_{n}({\bm{{\bm{\theta}}}}) &= \int  \big[f_{{\bm{{\bm{\theta}}}}}(x)B'(f_{{\bm{{\bm{\theta}}}}}(x)) - B(f_{{\bm{{\bm{\theta}}}}}(x))\big] dx - n^{-1} \sum_{i=1}^{n} B'(f_{{\bm{{\bm{\theta}}}}}(X_i))\\
    &= n^{-1} \sum_{i=1}^{n} V_{{\bm{{\bm{\theta}}}}}(X_i),
\end{aligned}
\end{equation}
where $V_{{\bm{{\bm{\theta}}}}}(t) = \int  \big[f_{{\bm{{\bm{\theta}}}}}(x)B'(f_{{\bm{{\bm{\theta}}}}}(x)) - B(f_{{\bm{{\bm{\theta}}}}}(x))\big] dx - B'(f_{{\bm{{\bm{\theta}}}}}(t))$. As the population analogue to Equation (\ref{eq9}), we define
\begin{equation}
\label{eq10}
  H({\bm{{\bm{\theta}}}}) = \int  \big[f_{{\bm{{\bm{\theta}}}}}(x)B'(f_{{\bm{{\bm{\theta}}}}}(x)) - B(f_{{\bm{{\bm{\theta}}}}}(x))\big] dx - \int B'(f_{{\bm{{\bm{\theta}}}}}(x))dG(x).
\end{equation}
We also define the following quantities. 
\begin{enumerate}
  \item The information function of the model: $I_{{\bm{{\bm{\theta}}}}}(x) = -\nabla_{\bm{{\bm{\theta}}}} u_{{\bm{{\bm{\theta}}}}}(x)$.
  \item The covariance function of $T(X) = u_{{\bm{{\bm{\theta}}}}}(X)B''(f_{{\bm{{\bm{\theta}}}}}(X))f_{{\bm{{\bm{\theta}}}}}(X)$ under $G$, which has the form
  \begin{equation}
  \label{eq11}
    \begin{aligned}
K({\bm{\theta}}) &= \int u_{{\bm{\theta}}}(x) u^T_{{\bm{\theta}}}(x) [B''(f_{{\bm{\theta}}}(x))f_{{\bm{\theta}}}(x)]^2 dG(x) - \xi({\bm{\theta}})\xi^T({\bm{\theta}}),\\
\xi({\bm{\theta}}) &= \int u_{{\bm{\theta}}}(x)B''(f_{{\bm{\theta}}}(x))f_{{\bm{\theta}}}(x) dG(x).
    \end{aligned}
\end{equation}
 \item The function $J({\bm{\theta}})$, where
  \begin{equation}
    \label{eq12}
    \begin{aligned}
J({\bm{\theta}}) &= \int u_{{\bm{\theta}}}(x) u^T_{{\bm{\theta}}}(x) B''(f_{{\bm{\theta}}}(x))f^2_{{\bm{\theta}}}(x) dx \\
          & \quad + \int [ I_{{\bm{\theta}}}(x) - u_{{\bm{\theta}}}(x) u^T_{{\bm{\theta}}}(x)  h(x)]\\
          & \quad \quad \times [g(x) - f_{{\bm{\theta}}}(x)][B''(f_{{\bm{\theta}}}(x))f_{{\bm{\theta}}}(x)] dx
    \end{aligned}
\end{equation}
where $w(t) = B''(t) \cdot t$, $w'(t) = B''(t) + B'''(t) \cdot t$ and $$h(t) = \frac{w'(f_{{\bm{\theta}}}(t))f_{{\bm{\theta}}}(t)}{w(f_{{\bm{\theta}}}(t))}.$$
\end{enumerate}

Thereom \ref{thm:thm1} is provided under the set of assumptions given below. These may be viewed as generalizations of the conditions presented in \cite{basu2011} (which were designed specifically for the DPD class). The details of the proof are not presented here, as it mimics the approach of Theorem 9.2 of \cite{basu2011} exactly. 
\begin{enumerate}
    \item[(A1)] The distributions $F_{{\bm{\theta}}}$ of $X$ have common support, so that the set $\chi = \{x: f_{{\bm{\theta}}}(x) > 0\}$ is independent of ${\bm{\theta}}$. The distribution of $G$ is also supported on $\chi$, on which the corresponding density $g$ is greater than zero.
    \item[(A2)]  There is an open subset $\omega$ of the $s$-dimensional parameter space $\Omega$ containing the best fitting parameter ${\bm{\theta}}_{g}$ such that for almost all $x \in \chi$ and all ${\bm{\theta}} \in \omega$, the density $f_{{\bm{\theta}}}(x)$ is three times differentiable with respect to ${\bm{\theta}}$ and the third partial derivatives are continuous with respect to ${\bm{\theta}}$.
    \item[(A3)] The integrals $\int [f_{{\bm{\theta}}}(x) B'(f_{{\bm{\theta}}}(x)) - B(f_{{\bm{\theta}}}(x))] dx$ and $\int B'(f_{{\bm{\theta}}}(x)) g(x) dx$ can be differentiated three times with respect to ${\bm{\theta}}$ and the derivatives can be taken under the integral sign.
    \item[(A4)]  The $s \times s$ matrix $J(\bm{\theta})$, with its $(k,l)$ entry defined as
    \begin{equation}
        \label{eq13}
        J_{kl}({\bm{\theta}}) = \mathrm{E}_{g} \Big[\nabla_{kl} \Big\{ \int [ f_{{\bm{\theta}}}(x)B'(f_{{\bm{\theta}}}(x)) - B(f_{{\bm{\theta}}}(x))] dx -  B'(f_{{\bm{\theta}}}(x))\Big\}\Big],
    \end{equation}
    is positive definite. Here $\nabla_{kl}$ denotes the partial derivative of a function with respect to the $k$th and $l$th components of its argument and   ${E}_g$ represents the expectation under the density $g$. 
    \item[(A5)] There exists a function $M_{jkl}(x)$ such that
    \begin{equation*}
        \mid \nabla_{jkl} V_{{\bm{\theta}}}(x) \mid \  \leq M_{jkl}(x)\quad \forall \  {\bm{\theta}} \in \omega,
    \end{equation*}
    where $V_{{\bm{\theta}}}(x) = \int \big[ f_{{\bm{\theta}}}(x)B'(f_{{\bm{\theta}}}(x)) - B(f_{{\bm{\theta}}}(x)) \big] dx -  B'(f_{{\bm{\theta}}}(x))$ and $$\mathrm{E}_{g}[M_{jkl}(X)] = m_{jkl} < \infty \text{\quad for all } j, k \text{ and } l.$$
\end{enumerate}

\begin{thm}
\label{thm:thm1}
Assuming that conditions A1-A5 hold,
\begin{enumerate}
  \item The estimating equation given by Equation (\ref{eq3}) has a consistent sequence of roots $\hat{{\bm{\theta}}} = \hat{{\bm{\theta}}}_n$ and
  \item $\sqrt n (\hat{{\bm{\theta}}} - {\bm{\theta}}_g)$ has an asymptotic multivariate normal distribution with (vector) mean zero and covariance matrix $J^{-1}KJ^{-1}.$
\end{enumerate}
\end{thm}


\subsection{Influence function and standard error}
\label{subsec:if}
Recalling our formulation of any minimum Bregman divergence based estimator (say, $T_B$) as an M-estimator, we observe that its influence function is given by
\begin{equation}
\label{eq14}
    IF(y, T_{B}, G) = J^{-1} [u_{{\bm{\theta}}}(y)B''(f_{{\bm{\theta}}}(y))f_{{\bm{\theta}}}(y) - \xi ],
\end{equation}
where $\xi$ and $J$ are given in Equations (\ref{eq11}) and (\ref{eq12}) repectively. These quantities get simplified further in case the true distribution $G$ belongs to the model under consideration. Assuming $J$ and $\xi$ to be finite, the influence function turns out to be bounded whenever the quantity $u_{{\bm{\theta}}}(y)B''(f_{{\bm{\theta}}}(y))f_{{\bm{\theta}}}(y)$ is bounded in $y$ (or, equivalently, when $u_\theta(y) w(f_\theta(y))$ is bounded in $y$).  This is true, for example, in case of all members of the DPD family with all $\alpha > 0$, and for most standard parametric families (including the normal location-scale family). In case of MEWDE($\beta$), the influence function is immediately seen to be
\begin{equation*}
    IF(y, T_{\beta}, G) = J^{-1} \Big[u_{{\bm{\theta}}}(y)\big[ 1 - \exp (f_{{\bm{\theta}}}(y)/\beta) \big] - \xi \Big].
\end{equation*}
In Figure \ref{fig:f3}, the influence functions of MEWDE($\beta$) for the estimation of the normal mean when $\sigma=1$ are plotted; for all $\beta > 0$ considered here, we note their bounded redescending nature.
\begin{figure}[h]
  \centering
  \includegraphics[width = 0.8\textwidth]{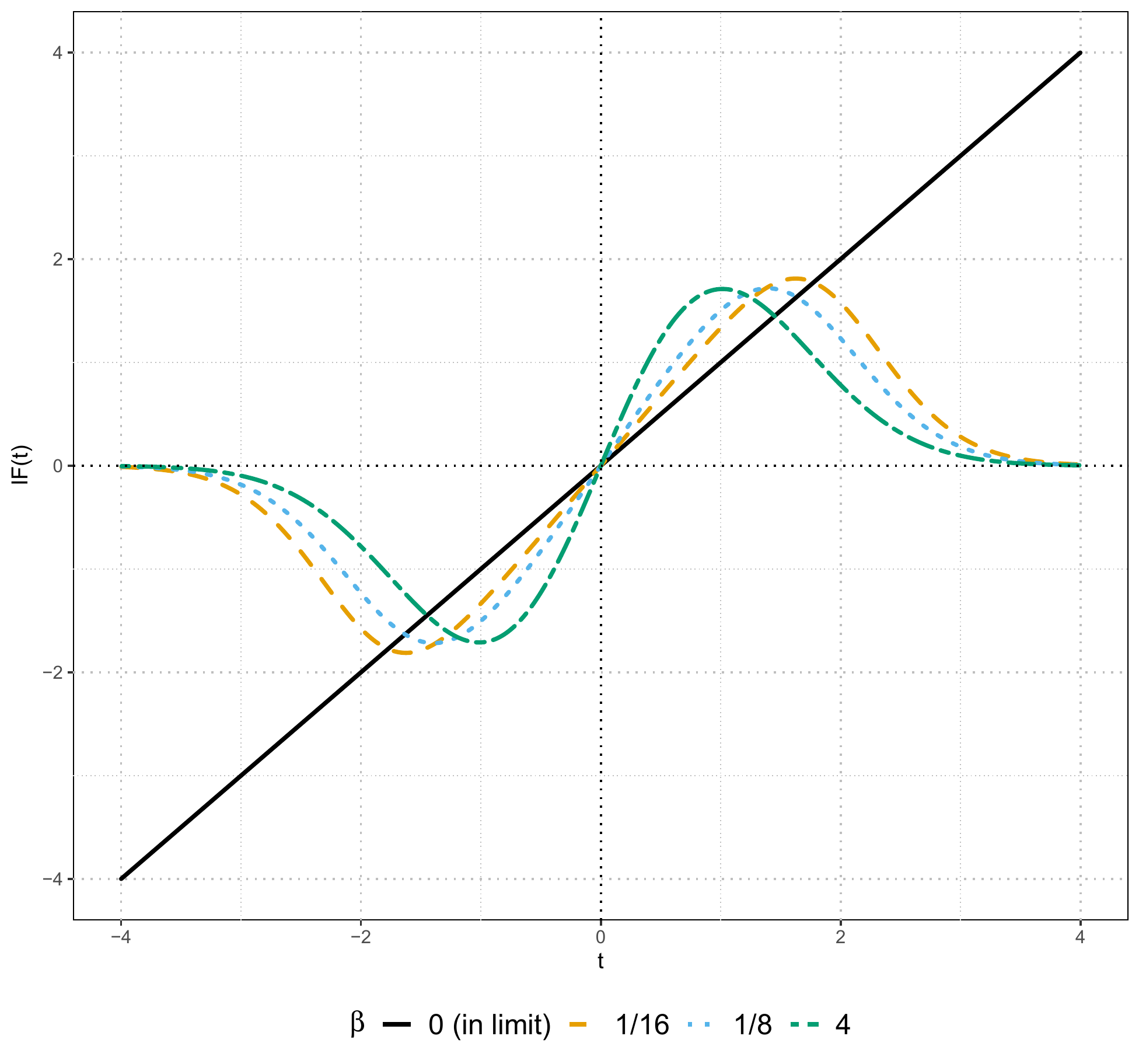}
  \caption{IFs for MEWDE($\beta$) for $\hat{\mu}$ of $N(\mu, 1)$ family at $N(0, 1)$ distribution.}\label{fig:f3}
\end{figure}

\begin{remark}
The asymptotic variance of $\sqrt n$ times the MEWDE($\beta$) can be consistently estimated in a {\em sandwich fashion} by using the above influence function, as in \cite{huber2009}. Let $K_i = u_{{\bm{\theta}}}(X_i)(1-\exp(-f_{{\bm{\theta}}}(X_i)/\beta)) - \xi({\bm{\theta}})$ and let $\hat{K_i}$ be the corresponding quantity evaluated at $\hat{{\bm{\theta}}}$, with the empirical distribution $G_{n}$ plugged in place of $G$. Let $\hat{K} = (n-1)^{-1} \sum_i (\hat{K}_{i}\hat{K}^{T}_{i})$. Similarly, we obtain $\hat{J}$ from $J$ by replacing ${\bm{\theta}}$ by $\hat{{\bm{\theta}}}$, with $G_n$ plugged in place of $G$. Then, the asymptotic variance of $\sqrt n$ MEWDE($\beta$) can be consistently estimated by $\hat{J}^{-1}\hat{K}\hat{J}^{-1}.$ Consistent estimators of the asymptotic variance of this estimator can also be obtained by the jackknife and bootstrap techniques. Again, this technique can be extended to consistently estimate the asymptotic variance of any minimum Bregman divergence based estimator for a given $B$ function. 
\end{remark} 

\section{Estimation for independent and identical data}
\label{sec:ihd}
\subsection{Introduction}
\label{iid:intro}
When the true distribution $G$ belongs to the model, i.e. $G = F_{{\bm{\theta}}}$ for some ${\bm{\theta}} \in \Omega$, the formulae for $J$, $K$ and $\xi$, in case of MEWDE($\beta$), is as follows. 
\begin{equation}
\label{eq15}
\begin{aligned}
  J &= \int  u_{{\bm{\theta}}}(x) u^{T}_{{\bm{\theta}}}(x) [1-\exp(-f_{{\bm{\theta}}}(x)/\beta)] f_{{\bm{\theta}}}(x) dx, \\
  K &= \int  u_{{\bm{\theta}}}(x) u^{T}_{{\bm{\theta}}}(x) [1-\exp(-f_{{\bm{\theta}}}(x)/\beta)]^{2} f_{{\bm{\theta}}}(x)  dx \quad - \quad \xi\xi^T, \\
  \xi &= \int u_{{\bm{\theta}}}(x) [1-\exp(-f_{{\bm{\theta}}}(x)/\beta)] f_{{\bm{\theta}}}(x) dx.\\
\end{aligned}
\end{equation}
As $\beta \rightarrow 0$, $J$ and $K$ both tend to the Fisher information matrix. We use Equation (\ref{eq15}) to compute asymptotic relative efficiencies of MEWDE($\beta$), which indicate how much efficiency is lost, relative to the maximum likelihood estimator, under the pure model. Along the lines of \cite{basu1998}, we consider examples of some specific parametric families.

\subsection{Simulation scheme}
\label{subsec:simscheme}

Here we consider different parametric families and compute the MEWDEs of the model parameters under different scenarios using simulated data and compare them with the corresponding MDPDEs. At the first stage we compute, for a fixed parametric family of densities 
$\mathscr{F}_\theta = \{f_\theta: \theta \in \Omega \subset \mathrm{R}^p\}$, the empirical mean square errors (MSEs) of the parameter estimates -- for several members of both the MDPDE and the MEWDE classes -- under pure data generated from the given parametric model.  Then we identify several sets of combinations $(\alpha_0, \beta_0)$, the tuning parameters of the two families,  for which the empirical MSEs of MDPDE($\alpha_0$) and MEWDE($\beta_0$) are approximately equal.  Subsequently we generate data from contaminated model distributions having densities of the form
$$h(x) = (1 - \epsilon) f_{\theta_0}(x) + \epsilon v(x),$$ where $\epsilon$ is the contaminating proportion, $v(x)$ is a suitable contaminating density, but $\theta_0$ is still the target parameter. Now we compare the MSEs of  MDPDE ($\alpha_0$) and MEWDE($\beta_0$), with an aim to determine which one of these two, which are close in terms of model efficiency, have better outlier stability.  Unless otherwise mentioned, we have used samples of size $n  = 200$, and for each scenario we have replicated the sample $r = 2000$ times. The finite sample relative efficiency (FSRE) of the MDPDE is defined to be the ratio of MSE(MLE) to MSE(MDPDE); similarly for the MEWDE. The relevant R codes are presented in the Online Supplement.


\subsection{Simulation study: mean of univariate normal}
\label{subsec:simNormMean}
Taking $f_{{\bm{\theta}}}$ to be the density function for $N(\mu, \sigma^2)$ with $\sigma^2$ known, using Theorem \ref{thm:thm1} and Equation (\ref{eq15}), one can compute and compare theoretical asymptotic relative efficiencies (AREs) of both the MEWDE($\beta$) and MDPDE($\alpha$) with respect to the MLE (see Table \ref{tab:tab1}).
    \begin{table}[h]
    \centering
    \caption{AREs of MDPDE and MEWDE of $\mu$ for N($\mu$, 1).}
    \label{tab:tab1}
    \begin{tabular}{ccc}
      \toprule
      Tuning par. ($k$) & ARE(MDPDE($k$)) & ARE(MEWDE($k$)) \\ \midrule
      0.001 & 1.000 & 0.996 \\
      0.004 & 1.000 & 0.987 \\
      0.016 & 1.000 & 0.955 \\
      0.062 & 0.995 & 0.867 \\
      0.250 & 0.941 & 0.741 \\
      1.000 & 0.650 & 0.676 \\
      4.000 & 0.216 & 0.656 \\
      \bottomrule
    \end{tabular}
        \end{table}
As both $\alpha$ and $\beta$ move away from zero, the efficiency of MDPDE($\alpha$) decreases slowly for a brief initial period, but then drops much more rapidly as compared to MEWDE($\beta$), as is seen in Table \ref{tab:tab1}. For a simulation-based comparison of the DPD and EWD classes, we follow the scheme outlined in Section \ref{subsec:simscheme}, where the true distribution is $N(0,1)$ and the contaminating distribution is $N(\mu_c,1)$. We have carried out simulation studies for $\mu_c = 3 \text{ and } 5$ and estimated the mean parameter under the $N(\mu, 1)$ model.  Our findings are presented in Table \ref{tab:tab2}.
\begin{table}[h]
  \centering
  \caption{FSRE's of MDPDE(denoted D($\alpha$)) and MEWDE(denoted ) with respect to MLE.  \\ \scriptsize  Figures in bold denote best FSRE in that contamination scheme.}
  \label{tab:tab2}
  \begin{tabular}{lccccccccc}
    \toprule
     & \multicolumn{4}{c}{$\mu_c = 3$} && \multicolumn{3}{c}{$\mu_c = 5$} \\
    \cline{2-5} \cline{7-9} \multicolumn{1}{c}{$\hat{\mu}$}  & $\epsilon=0$ & $\epsilon=0.05$ & $\epsilon=0.10$ & $\epsilon=0.20$ && $\epsilon=0.05$ & $\epsilon=0.10$ & $\epsilon=0.20$ \\ \midrule
MLE & $\bm{1}$ & $1$ & $1$ & $1$ && $1$ & $1$ & $1$ \\\midrule
D(0.05) & $0.996$ & $1.358$ & $1.358$ & $1.250$ && $2.635$ & $2.568$ & $2.059$ \\
E(0.001) & $0.996$ & $1.791$ & $1.806$ & $1.495$ && $12.326$ & $31.141$ & $52.727$ \\\midrule
D(0.1) & $0.956$ & $2.567$ & $3.027$ & $2.552$ && $10.966$ & $23.342$ & $29.168$ \\
E(0.004) & $0.954$ & $3.409$ & $4.863$ & $4.221$ && $\bm{13.509}$ & $\bm{43.633}$ & $\bm{140.125}$ \\\midrule
D(0.43) & $0.871$ & $3.592$ & $6.075$ & $6.213$ && $12.106$ & $38.450$ & $115.779$ \\
E(0.063) & $0.867$ & $\bm{4.003}$ & $7.947$ & $9.664$ && $12.356$ & $40.769$ & $137.303$ \\\midrule
D(0.74) & $0.749$ & $3.693$ & $8.567$ & $13.500$ && $10.495$ & $34.680$ & $117.038$ \\
E(0.25) & $0.747$ & $3.763$ & $\bm{9.075}$ & $15.557$ && $10.525$ & $34.861$ & $118.461$ \\\midrule
D(0.98) & $0.666$ & $3.428$ & $8.837$ & $17.716$ && $9.304$ & $30.871$ & $105.076$ \\
E(4) & $0.666$ & $3.430$ & $8.861$ & $17.852$ && $9.304$ & $30.871$ & $105.129$ \\\midrule
$L_2$ & $0.659$ & $3.401$ & $8.821$ & $\bm{17.977}$ && $9.206$ & $30.553$ & $104.007$ \\ \bottomrule
  \end{tabular}
\end{table}
The first column presents the $(\alpha_0, \beta_0)$ combinations used in this example, and the second column indicates how close the corresponding MSEs are. The following important observations can be made from the figures of Table 2. 
\begin{enumerate}

\item For uncontaminated data, the MLE is the most efficient estimator, as it should be. 

\item Under even a slight contamination there is a severe degradation in performance of the MLE, and the other two estimators quickly overtake it. 

\item Generally, as the proportion of contamination increases, larger tuning parameters give better performance (on account of their stronger downweighting). However, this improvement is not absolute. Generally, with increasing tuning parameter,  the performance of the estimators reach a peak at some moderate value of the tuning parameter, and thereafter drops again.  

\item If the contaminating distribution is far separated from the target distribution, smaller values of tuning parameters are sufficient to provide good outlier stability. 

\item However, the most important observation for us is that in all the pairs considered here having comparable MSEs under pure data, the MEWDE beats the MDPDE, sometimes quite soundly, under contaminated scenarios. 

\end{enumerate}

\begin{figure}[h]
  \centering
  \includegraphics[width = 0.8\textwidth]{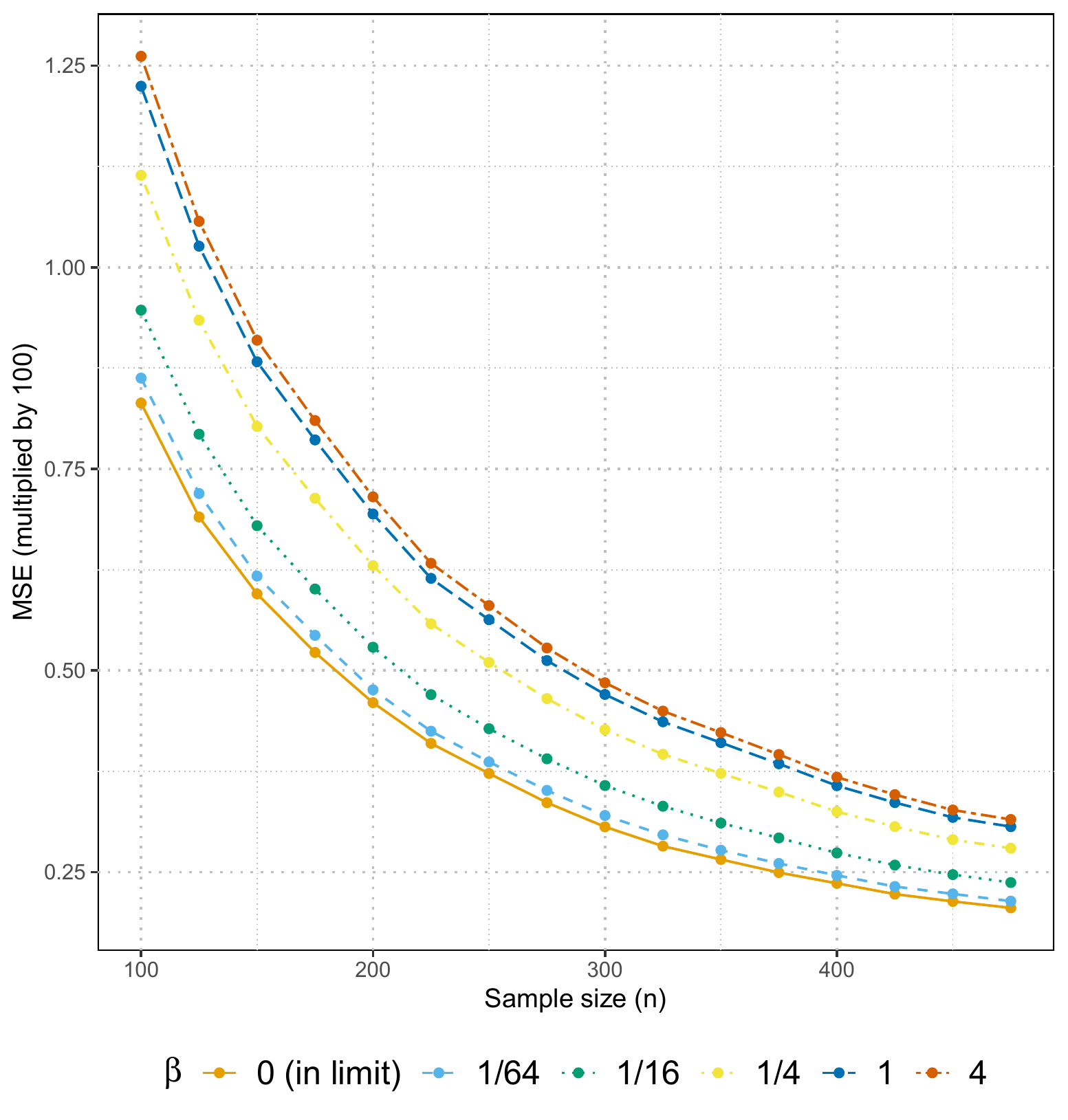}
  \caption{MSE of MEWDE($\beta$) of $\mu$ under the $N(\mu, 1)$ model for pure $N(0, 1)$ data.}\label{fig:f4}
\end{figure}

In Figures 4 and 5 we graphically present the MSEs of different members of the MEWDE($\beta$) class for the indicated pure normal data and contaminated normal data situations over a sequence of sample sizes. Figure 4 clearly shows the hierarchical relation between increasing $\beta$ and increasing MSE. In Figure 5 it may be seen that the optimal MSE is at an intermediate value of $\beta$. If we brought the contaminating mean closer, or pushed up the contaminating proportion, a higher value of $\beta$ would be required for the optimal solution.

\begin{figure}[]
  \centering
  \includegraphics[width = 0.8\textwidth]{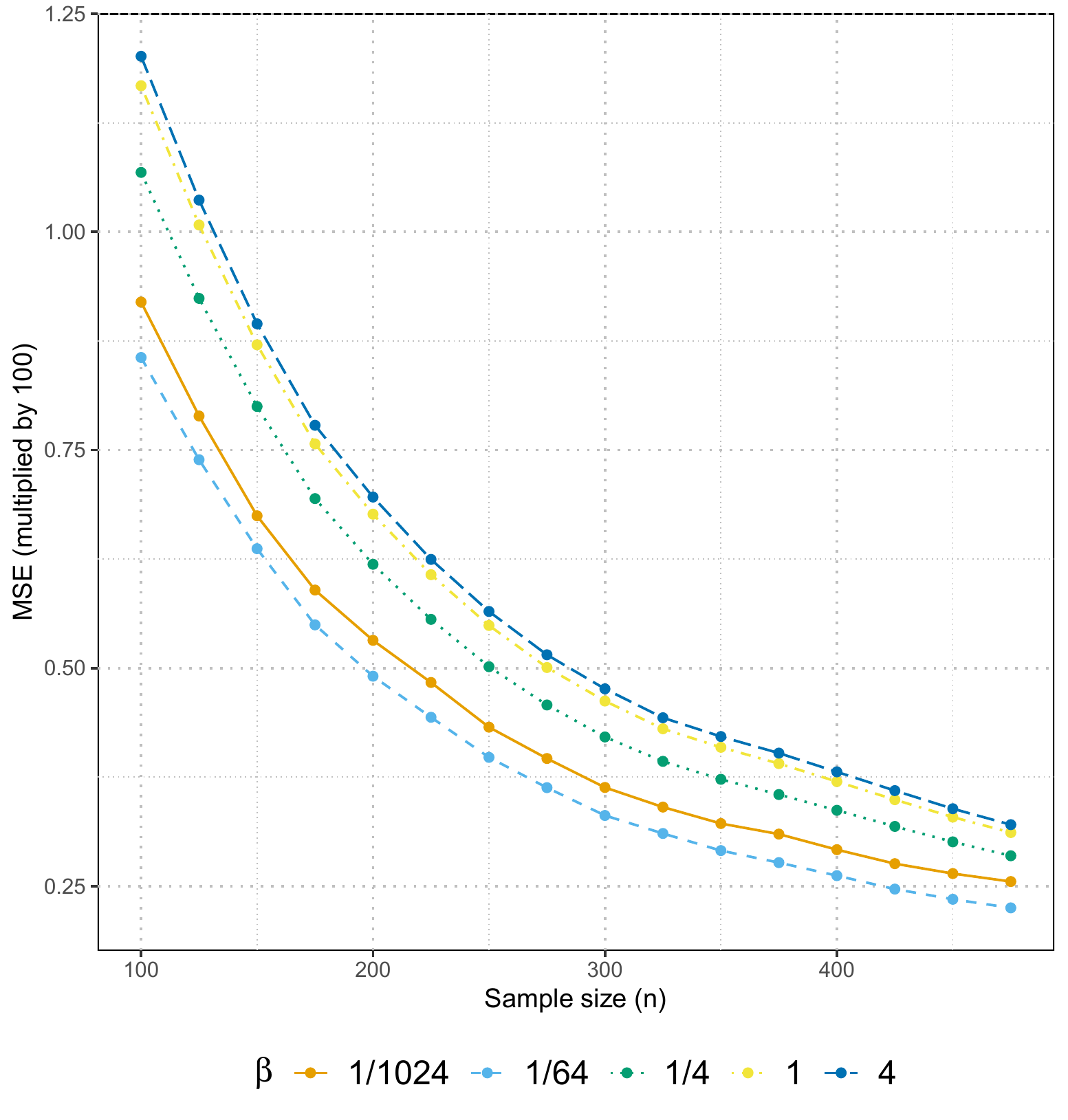}
  \caption{MSE of MEWDE($\beta$) for data from $0.95 N(0, 1) + 0.05 N(5,1)$.}\label{fig:f4}
\end{figure}

%
%

\subsection{Simulation study: standard deviation of univariate normal}
\label{subsec:simNormSD}
We compare the robustness of the competing MDPDE($\alpha$) and MEWDE($\beta$) classes in the context of estimating $\sigma$ when data come from a contaminated normal distribution given by $(1-\epsilon) N(0, 1) + \epsilon N(0, \sigma^2_{c}),$  where $\epsilon$ is the contamination proportion and $\sigma^2_{c}$ is the variance of the contaminating distribution; the model is the $N(0, \sigma^2)$ model and the target parameter is 1. The observations from the results reported in Table 3 are very similar to those for Table  2. Once again, we note that within each (efficiency-wise) equivalent pair, the MEWDE beats the MDPDE in each single case. 
\begin{table}[htp]
  \centering
  \caption{FSRE's of MDPDE(denoted D($\alpha$)) and MEWDE(denoted E($\beta$)) with respect to the MLE for the normal scale model.\\ \scriptsize  Figures in bold denote best FSRE in that contamination scheme.}
  \label{tab:tab4}
  \begin{tabular}{lccccccccc}
    \toprule
     & \multicolumn{4}{c}{$\sigma_c = 3$} && \multicolumn{3}{c}{$\sigma_c = 5$} \\
    \cline{2-5} \cline{7-9} \multicolumn{1}{c}{$\hat{\sigma}$}  & $\epsilon=0$ & $\epsilon=0.05$ & $\epsilon=0.10$ & $\epsilon=0.20$ && $\epsilon=0.05$ & $\epsilon=0.10$ & $\epsilon=0.20$ \\ \midrule
MLE & $\bm{1}$ & $1$ & $1$ & $1$ && $1$ & $1$ & $1$ \\ \midrule
D(0.098) & $0.970$ & $2.980$ & $2.476$ & $1.789$ && $10.270$ & $5.478$ & $2.356$ \\
E(0.001) & $0.971$ & $5.670$ & $5.017$ & $2.806$ && $40.979$ & $34.196$ & $10.234$ \\\midrule
D(0.177) & $0.873$ & $6.269$ & $6.338$ & $4.057$ && $38.358$ & $33.529$ & $13.105$ \\
E(0.004) & $0.872$ & $8.555$ & $10.786$ & $7.146$ && $\bm{60.167}$ & $79.181$ & $47.703$ \\\midrule
D(0.551) & $0.670$ & $7.950$ & $12.253$ & $10.299$ && $50.158$ & $72.823$ & $52.318$ \\
E(0.063) & $0.669$ & $\bm{8.658}$ & $\bm{14.871}$ & $\bm{13.873}$ && $56.044$ & $\bm{95.436}$ & $\bm{82.653}$ \\\midrule
D(0.884) & $0.550$ & $7.103$ & $12.494$ & $12.443$ && $43.073$ & $67.906$ & $55.659$ \\
E(0.5) & $0.549$ & $7.161$ & $12.772$ & $12.907$ && $43.559$ & $69.896$ & $58.191$ \\\midrule
D(0.983) & $0.526$ & $6.839$ & $12.217$ & $12.463$ && $41.136$ & $64.981$ & $53.813$ \\
E(4) & $0.526$ & $6.844$ & $12.247$ & $12.512$ && $41.167$ & $65.181$ & $54.046$ \\\midrule
$L_2$ & $0.522$ & $6.801$ & $12.164$ & $12.453$ && $40.825$ & $64.499$ & $53.474$ \\ \bottomrule
  \end{tabular}
\end{table}


\subsection{Simulation study: mean of exponential}
\label{subsec:simExp}
We compare the robustness of the competing MDPDE($\alpha$) and MEWDE($\beta$) classes in the context of estimating the mean parameter $\lambda$ of the exponential model, when data come from a contaminated distribution given by $(1-\epsilon) E(1) + \epsilon E(\lambda_c),$  where $\epsilon$ is the contamination proportion and $E(\lambda)$ denotes an exponential distribution with mean $\lambda$. Here $\lambda_c$ is the mean of the contaminating exponential distribution and the target parameter value is $\lambda$ (fixed at $1$). 
\begin{table}[htp]
  \centering
  \caption{FSRE's of MDPDE(denoted D($\alpha$)) and MEWDE(denoted E($\beta$)) with respect to the MLE for the exponential model.\\ \scriptsize  Figures in bold denote best FSRE in that contamination scheme.}
  \label{tab:tab9}
  \begin{tabular}{lccccccccc}
    \toprule
     & \multicolumn{4}{c}{$\lambda_c = 3$} && \multicolumn{3}{c}{$\lambda_c = 5$} \\
    \cline{2-5} \cline{7-9} \multicolumn{1}{c}{$\hat{\lambda}$}  & $\epsilon=0$ & $\epsilon=0.05$ & $\epsilon=0.10$ & $\epsilon=0.20$ && $\epsilon=0.05$ & $\epsilon=0.10$ & $\epsilon=0.20$ \\ \midrule
MLE & $\bm{1}$ & $1$ & $1$ & $1$ && $1$ & $1$ & $1$ \\ \midrule
D(0.153) & $0.904$ & $1.592$ & $1.858$ & $1.771$ && $3.217$ & $3.438$ & $2.746$ \\
E(0.004) & $0.905$ & $1.766$ & $2.230$ & $2.101$ && $4.652$ & $6.002$ & $4.658$ \\\midrule
D(0.440) & $0.696$ & $1.753$ & $2.774$ & $3.099$ && $4.810$ & $7.896$ & $7.755$ \\
E(0.063) & $0.694$ & $\bm{1.824}$ & $\bm{3.112}$ & $3.655$ && $\bm{5.321}$ & $\bm{10.115}$ & $\bm{11.055}$ \\\midrule
D(0.844) & $0.525$ & $1.496$ & $2.752$ & $3.617$ && $4.237$ & $8.236$ & $9.793$ \\
E(1.000) & $0.525$ & $1.498$ & $2.766$ & $\bm{3.656}$ && $4.245$ & $8.282$ & $9.924$ \\\midrule
D(0.989) & $0.492$ & $1.420$ & $2.669$ & $3.615$ && $4.022$ & $7.958$ & $9.733$ \\
E(16.00) & $0.492$ & $1.420$ & $2.669$ & $3.616$ && $4.022$ & $7.958$ & $9.734$ \\\midrule
$L_2$ & $0.490$ & $1.414$ & $2.662$ & $3.613$ && $4.006$ & $7.935$ & $9.723$ \\\bottomrule
  \end{tabular}
\end{table}
Again the observations are similar to those of Tables 2 and Table 3. Once again the MEWDE is equivalent or better than the MDPDE in each single case.

\subsection{Modeling real life data: Shoshoni rectangles}
\label{shoshoni}
Data on Shoshoni rectangles presented and analyzed by \cite{hettmansperger2010} are studied here. Data on twenty width to length ratios of beaded rectangles found in baskets used by Shoshonis are given in Table \ref{tab:t5}.
\begin{table}[h]
  \centering
   \caption{Width to Length Ratios of Rectangles}\label{tab:t5}
   \begin{tabular}{|cccccccccc|}
     \hline
     0.553 & 0.570& 0.576& 0.601& 0.606& 0.606& 0.609 &0.611 &0.615& 0.628 \\
     0.654 &0.662& 0.668 &0.670 &0.672 &0.690 &0.693 &0.749 &0.844& 0.933 \\
     \hline
   \end{tabular}
\end{table}
First, we examine the histogram of the observations (see Figure \ref{fig:f5}) --- it is seen that 3 observations (colored in red) are well-separated from the `main body' made up of the remaining observations. We note from Figure \ref{fig:f6} that the Q-Q plot generated by all twenty observations (suitably centered and scaled) gives us more evidence to claim that the 3 largest observations are (possibly) outliers. On the other hand, by studying the Q-Q plot generated from the `outlier deleted' dataset (again  suitably centered and scaled), we are motivated to model the dataset using a normal distribution with  unknown mean $\mu$ and standard deviation $\sigma$ (i.e., $\bm{\theta} = (\mu, \sigma)^T$). 

\begin{center}
    \begin{figure}[htp]
    \captionsetup{justification=centering}
  \centering
  \includegraphics[width=0.8\textwidth]{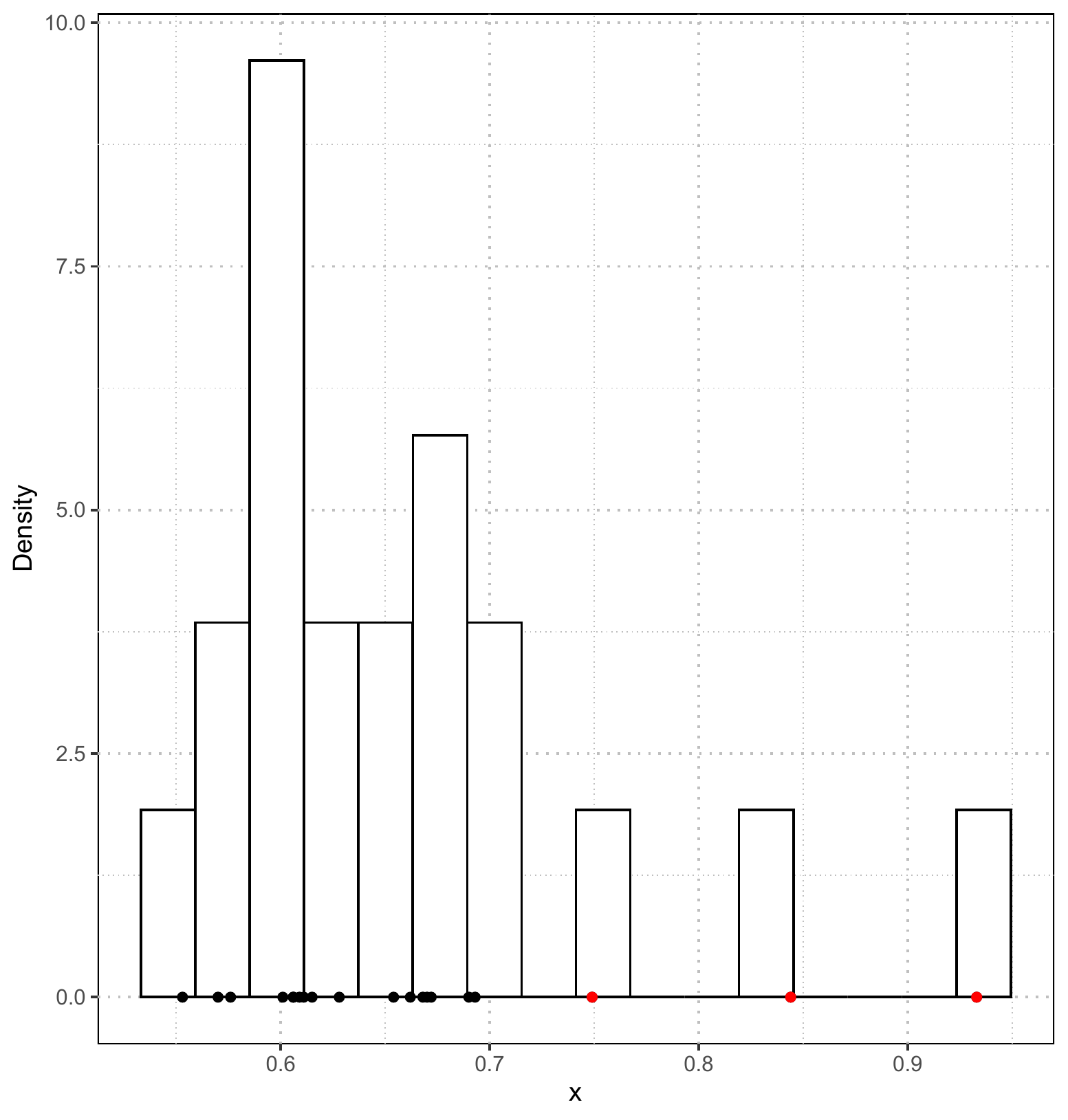}
  \caption{Histogram of values presented in Table \ref{tab:t5} with (suspect) outliers colored red.}\label{fig:f5}
\end{figure}
\end{center}
Based on the kernel based estimate of the density, we observe there is a mildly bimodal structure present in the main body of the dataset.
\begin{figure}[htp]
  \centering
  \includegraphics[width=\textwidth]{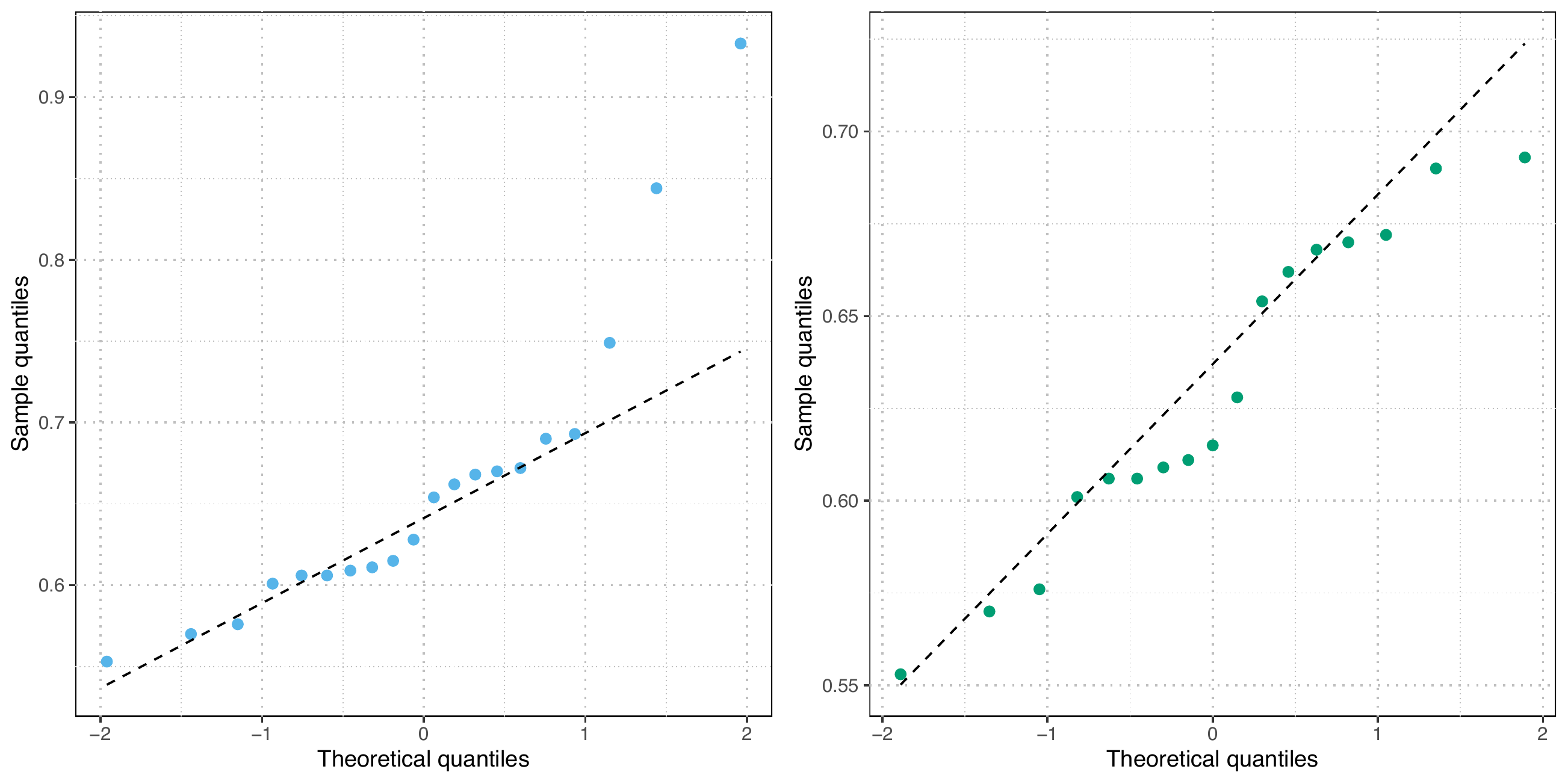}
  \caption{(L): Q-Q plot for complete data. (R): Q-Q plot for outlier deleted data.}\label{fig:f6}
\end{figure}
However, the Shapiro-Wilk test (Shapiro and Wilk (1965)), applied to the outlier deleted data, fails to reject the the null hypothesis that the outlier-deleted data are generated by a normal distribution. Indicating the full data maximum likelihood estimates by ML and the outlier deleted ones by ML+D, we get, under the normal model, $\hat{\mu}_{ML} = 0.660$ and $\hat{\sigma}_{ML} =0.093$; the outlier deleted estimates are
$\hat{\mu}_{ML+D} = 0.628$ and $\hat{\sigma}_{ML+D} =0.043$.  
Thus the outliers have a moderate effect on the mean, but a substantial effect on the scale parameter. However all the different tuning parameters for    DPD and EWD used in this case produce stable estimators and outlier resistant fits to the full data. As the outliers are quite distant from the majority of the data, small values of tuning parameters appear to sufficient in either case. 

\begin{figure}[htp]
  \centering
  \includegraphics[height=0.425\textheight, width=\textwidth]{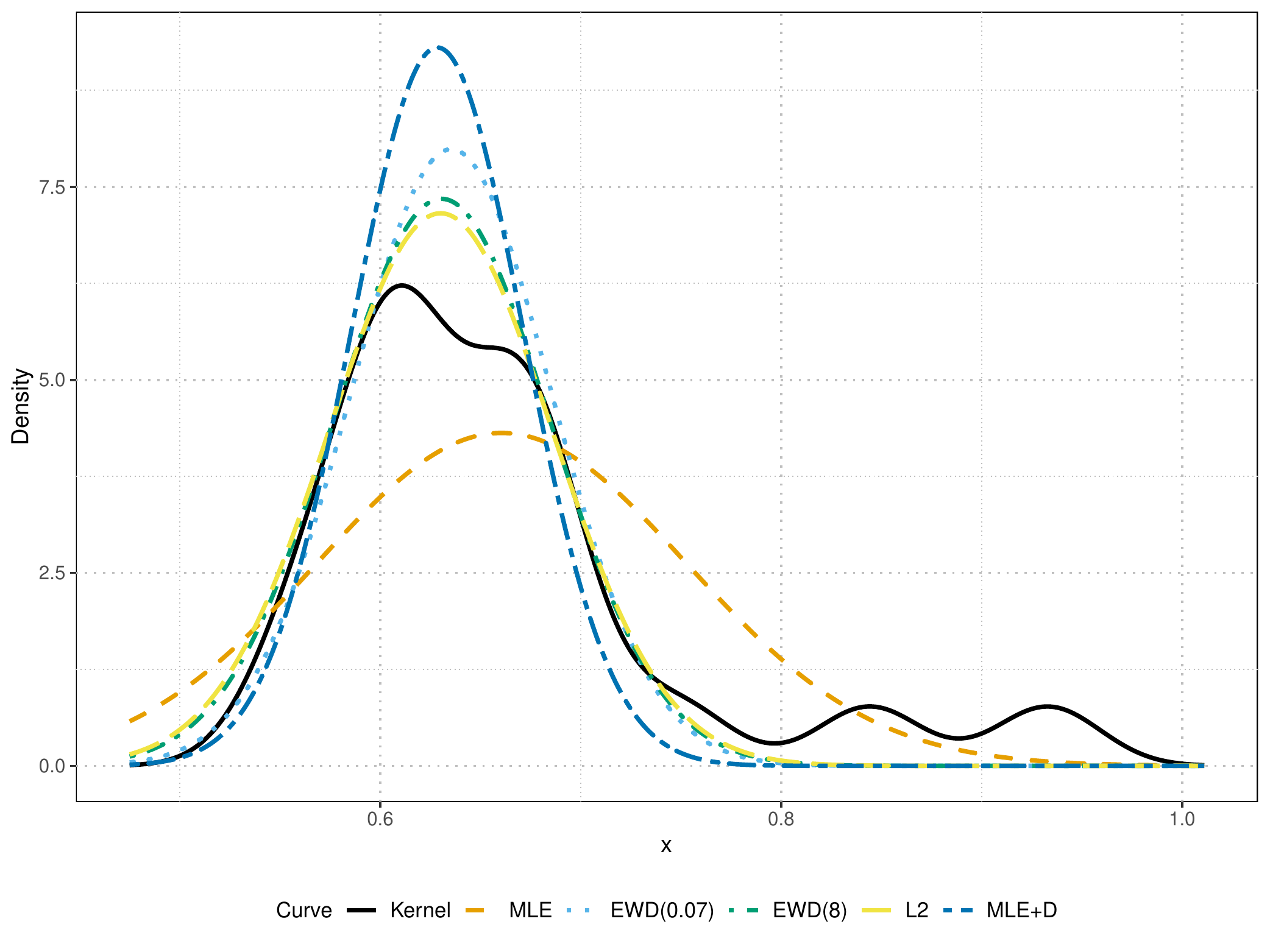}
  \caption{Density estimates for Shoshoni rectangles using MEWDEs.}\label{fig:f7}
\end{figure}

\begin{figure}[htp]
  \centering
  \includegraphics[height=0.425\textheight, width=\textwidth]{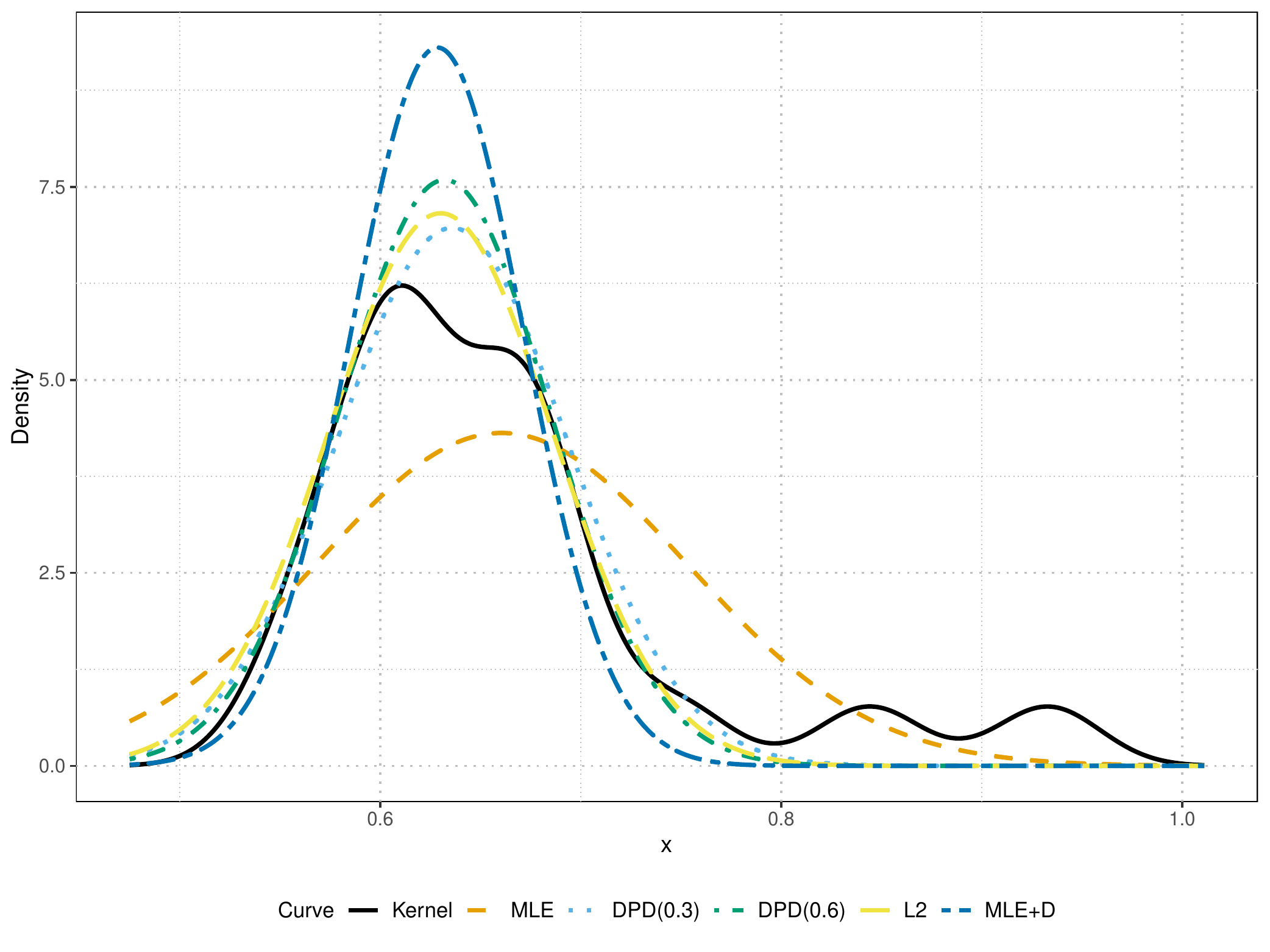}
  \caption{Density estimates for Shoshoni rectangles using MDPDEs.}\label{fig:f8}
\end{figure}

\subsection{Modeling real life data: Drosophila data}
\label{drosophila}
For data generally well modeled by the Poisson distribution, we choose to compare the performance of the MEWDE and MDPDE in the context of data on fruit flies (see \citep{woodruff1984chemical}).

\begin{table}[h]
\centering
\caption{Fitted frequencies for data using MLE, MDPDE(D($\alpha$)) and MEWDE(E($\beta$)).}
\label{tab:tab6}
\begin{tabular}{lccccccc}
\toprule
Count & 0 &1 & 2&3& 4 & $\geq 5$ & $\hat{\lambda}$ \\

Observed & 23 & 7 & 3 & 0 &0 & 1 (91) & -- \\
\midrule
MLE. & $1.596$ & $4.882$ & $7.467$ & $7.613$ & $5.822$ & $6.620$ & $3.059$ \\
\midrule
D($0.10$) & $22.981$ & $9.002$ & $1.763$ & $0.230$ & $0.023$ & $0.002$ & $0.392$ \\
D($0.50$) & $23.375$ & $8.759$ & $1.641$ & $0.205$ & $0.019$ & $0.002$ & $0.375$ \\
D($0.75$) & $23.549$ & $8.649$ & $1.588$ & $0.194$ & $0.018$ & $0.001$ & $0.367$ \\
\midrule
E($0.001$) &$22.894$ & $9.055$ & $1.791$ & $0.236$ & $0.023$ & $0.002$ & $0.396$ \\
E($0.02$)& $22.614$ & $9.222$ & $1.880$ & $0.256$ & $0.026$ & $0.002$ & $0.408$ \\
E($0.25$)& $23.712$ & $8.545$ & $1.540$ & $0.185$ & $0.017$ & $0.001$ & $0.360$ \\
\midrule
$L_2$ & $23.609$ & $8.611$ & $1.570$ & $0.191$ & $0.017$ & $0.001$ & $0.365$ \\
\midrule
MLE + D  & $22.93$ & $9.03$ & $1.78$ & $0.23$ &$0.02$ & $0$ & $0.39$ \\
\bottomrule
\end{tabular}
\end{table}

In this experiment male flies were sprayed with a certain level of a chemical to be screened, and then made to mate with unexposed females. The response, for each father fly, was the number of daughter flies having a recessive lethal mutation in the X-chromosome. The frequencies of these responses (presented in the first row of Table \ref{tab:tab6}) are modeled as Poisson variables, and the estimates of the Poisson mean parameter $\lambda$ (as well as the estimated frequencies), using several members of the DPD and EWD families, the MLE and the MLE+D (outlier deleted MLE) are presented in Table 6. The single extreme value at 91 is treated as the obvious outlier. Both set of estimators have comparable (satisfactory) performance. 


\subsection{Tuning parameter selection}
\label{ihd:tune}
It is clear that in doing estimation using the EWD, small values of $\beta$ provide greater model efficiency, while large values of $\beta$ provide greater outlier stability and protection against small model violations. Given any real data set we must choose the ``optimal'', data-based  tuning parameter $\beta$ so that the procedure has the right amount of balance as is necessary for the data set in  question. Here we follow the approach of \cite{warwick2005} to derive the optimal estimate of the tuning parameter. This approach constructs an empirical estimate of the mean square error as a function of the tuning parameter (and a pilot estimator). The empirical estimate of the mean square error MSE$_\beta$, as a function of the tuning parameter $\beta$ and a pilot estimator ${\bm{\theta}}^P$ is given by
\begin{equation*}
  \widehat{MSE_\beta}({\bm{\theta}}^P)=\left(\widehat{{\bm{\theta}}}_{\beta}-{\bm{\theta}}^{P}\right)^{T} \left(\widehat{{\bm{\theta}}}_{\beta}-{\bm{\theta}}^{P}\right)+ n^{-1} \text{tr}\left(J_{\beta}^{-1}\left(\widehat{{\bm{\theta}}}_{\beta}\right) K_{\beta}\left(\widehat{{\bm{\theta}}}_{\beta}\right) J_{ \beta}^{-1}\left(\widehat{{\bm{\theta}}}_{\beta}\right)\right),
\end{equation*}
where $J$ and $K$ are the terms defined in Equation (\ref{eq15}), $\widehat{\bm{\theta}}_\beta$ is the MEWDE($\beta$) and tr($\cdot$) denotes the trace of a matrix. By minimizing this objective function over the tuning parameter, we get a data driven `optimal' estimate of the tuning parameter. \cite{warwick2005} propose the minimum $L_2$ estimator as the pilot estimator in the above calculation, as it has strong robustness properties. 

For the data on Shoshoni rectangles presented in Section \ref{shoshoni}, we implement the tuning parameter selection algorithm detailed above. The normal distribution with unknown mean and standard deviation parameters is used to model this data. The optimal tuning parameter is found to be $\beta_{OPT} = 0.43$, and the associated estimated parameters are $\hat{\mu} = 0.63$ and $\hat{\sigma} = 0.05$. The corresponding (sample size-scaled) asymptotic mean-squared error is $5.07 \times 10^{-3}$.

A similar exercise is carried out using the Poisson distribution to model the data on Drosophila fruit flies presented in Section \ref{drosophila}. The optimal tuning parameter is found to be $\beta_{OPT} = 0.08$, and the estimated mean parameter is given by $\hat{\lambda} = 0.377$. The corresponding (sample size-scaled) asymptotic mean-squared error is $0.46$. 

See \citep{basak2020optimal} for some other approaches to tuning parameter selection. 

\section{Estimation for independent and non-homogeneous data}
\subsection{Introduction}
\label{inhd:intro}
In this section, going beyond the i.i.d. situation, we extend our method to the case of data which are independent and share common parameters in their distribution but are not identically distributed.  \cite{ghosh2013} refer to such data as independent and non-homogeneous observations and we adhere to that nomenclature. Exploiting the robustness of our minimum distance procedure, we develop a general estimation method for handling such data.  We establish the asymptotic properties of the proposed estimator, and illustrate the benefits of our method in case of linear regression.

We assume that our observed data $Y_1,\ldots, Y_n$ are independent. For $i=1, 2, \ldots, n$ we have $Y_i \sim g_i$, where $g_i$ are possibly different densities with respect to some common dominating measure. We want to model $g_i$ by the family $\mathscr{F}_{i,{\bm{\theta}}} = \{F_i(\cdot; {\bm{\theta}}) \mid {\bm{\theta}} \in \Omega\}$ for each $i = 1, 2,\ldots, n.$ An estimate of the Bregman divergence between the density corresponding to the $i$-th data point and the associated model density given by
\begin{equation*}
  d_{B}(\hat{g}_i(\cdot), f_{i}(\cdot; {\bm{\theta}})).
\end{equation*}
Since our aim is to reach some `common' value of ${\bm{\theta}}$ (if it exists) which can be used to model each $g_i$ individually, it is  intuitive to minimize the average divergence between the data points and the models. Consequently, we minimize 
      \begin{equation*}
        n^{-1} \sum_{i=1}^n d_{B}(\hat{g}_i(\cdot), f_{i}(\cdot; {\bm{\theta}}))
      \end{equation*}
with respect to ${\bm{\theta}}$, where $\hat{g}_i$ is a non-parametric density estimate of $g$. 
%
As in the approach suggested by \cite{ghosh2013}, in presence of only one data point $Y_i$ from density $g_i$, the best density estimate of $g_i$ is taken to be the (degenerate) density which puts the entire mass on $Y_i$. Consequently, our objective function becomes $H_{n}({\bm{\theta}})= n^{-1} \sum_{i=1}^n V_{i}(Y_i, {\bm{\theta}})$, which can be simplified as
\begin{equation}
\label{eq16}
n^{-1} \sum_{i=1}^n \Big[ \int \Big\{ f_{i}(y; {\bm{\theta}}) B'(f_{i}(y; {\bm{\theta}})) - B(f_{i}(y; {\bm{\theta}})) \Big\} dy - B'(f_{i}(Y_i; {\bm{\theta}}))\Big].
\end{equation}
In case of the MEWDE($\beta$), the $B$ function is given by Equation (\ref{eq6}). Considering partial derivatives of Equation (\ref{eq16}) with respect to ${\bm{\theta}}$, we arrive at the estimating equation $\nabla_{\bm{\theta}} \sum_{i=1}^n  V_{i}(Y_i, {\bm{\theta}})  = 0$, which can be rewritten as
\begin{equation}
\label{eq17}
\sum_{i=1}^n \Big[ u_{i}(Y_i)w(f_{i}(Y_i; {\bm{\theta}})) - \int \Big\{ u_{i}(t)w(f_{i}(t; {\bm{\theta}}))f_{i}(t; {\bm{\theta}})dt \Big\} \Big] = 0,
\end{equation}
where $u_i(x) = \nabla_{{\bm{{\bm{\theta}}}}} \log(f_i(x, \bm{\theta}))$ is the likelihood score function of the density $f_i(x,\bm{\theta})$ used to model the $i$-th data point, and $w(t)=B''(t) \times t$. For MEWDE($\beta$),  $w(t)=1 - \exp(-t/\beta).$ We note that as $\beta \rightarrow 0$, the corresponding objective function becomes
\begin{equation*}
  \sum_{i=1}^n [-\log (f_{i}(Y_i, {\bm{\theta}}))],
\end{equation*}
and the associated estimating equation becomes
\begin{equation*}
  \sum_{i=1}^n u_i(Y_i, {\bm{\theta}}) = 0.
\end{equation*}
We arrive at the fact that the objective function given by Equation (\ref{eq16}) and estimating equation given by Equation (\ref{eq17}) are simple generalizations of the maximum likelihood score equation for independent and non-homogeneous data.

\begin{remark}
In terms of statistical functionals, the minimum Bregman divergence based
functional $T_B(G_1, . . . , G_n)$ for non-homogeneous observations is given by the
relation
\begin{equation*}
 T_B(G_1, . . . , G_n) = \underset{{\bm{\theta}} \in \Omega}{\mathrm{argmin}}\ n^{-1} \sum_{i=1}^n d_{B}({g}_i(\cdot), f_{i}(\cdot; {\bm{\theta}})).
\end{equation*}
Since we have already established that the Bregman divergence is a genuine divergence (in the sense that it is non-negative and attains its minimum if and only if the two arguments are identical), it follows that the functional $T_B(G_1, . . . , G_n)$ is Fisher consistent under the assumption of the identifiability of the model.
\end{remark}

\subsection{Asymptotic properties.}
\label{inhd:asymp}
We derive the asymptotic distribution of the minimum exponentially weighted divergence estimator $\hat{{\bm{\theta}}}_n$ defined by the relation
\begin{equation*}
\hat{{\bm{\theta}}}_n = \underset{{\bm{\theta}} \in \Omega}{\mathrm{argmin}}\ H_{n}({\bm{\theta}})
\end{equation*}
provided such a minimum exists, where $H_n(\theta)$ is as defined in Equation (\ref{eq16}). We will be working under the framework as discussed in Section \ref{inhd:intro}.  We
also assume that there exists a best fitting parameter of ${\bm{\theta}}$ which is independent of the index $i$ of the different densities and let us denote it by ${\bm{\theta}}_g$. It is important to note that this assumption is satisfied if all the true densities $g_i$ belong to the model family so that $g_i = f_i(\cdot; {\bm{\theta}})$ for some common ${\bm{\theta}}_0$, and in that case the best fitting parameter is that true parameter ${\bm{\theta}}_0$. We know that the minimum Bregman divergence based estimator $\hat{{\bm{\theta}}}_n$ is obtained as a solution of the estimating equation given by Equation  (\ref{eq17}); as per our definition, this equation is satisfied by the minimizer of $H_n({\bm{\theta}})$ as defined in Equation (\ref{eq16}). We now define, for $i= 1, 2, \ldots$
\begin{equation}
\label{eq18}
H^{(i)}({\bm{\theta}}) = \int \Big\{ f_{i}(y; {\bm{\theta}}) B'(f_{i}(y; {\bm{\theta}})) - B(f_{i}(y; {\bm{\theta}})) \Big\} dy - \int \Big\{ B'(f_{i}(y; {\bm{\theta}})) g_{i}(y) \Big\} dy,
\end{equation}
so that at the best fitting parameter (i.e., our target parameter value ${\bm{\theta}}_g$), we have
\begin{equation*}
\nabla H^{(i)}({\bm{\theta}}_g) = 0, \quad i=1, 2, \ldots
\end{equation*}
We also define, for each $i=1, 2, \ldots$, the $s \times s$ matrix $J^{(i)}$ whose $(k, l)$-th entry is given by
\begin{equation}
\label{eq19}
J_{kl}^{(i)} = E_{g_{i}}[\nabla_{kl} V_i (Y; {\bm{\theta}})],
\end{equation}
where $\nabla_{kl} V_i (Y; {\bm{\theta}}) = \frac{\partial^2 V_i (Y; {\bm{\theta}})}{\partial \theta_k \partial \theta_l}$ and $E_{g_i}(\cdot)$ denotes taking expectation under the distribution specified by $g_i$. We also define
\begin{equation}
\label{eq20}
\Psi_{n} = n^{-1} \sum_{i=1}^{n} J^{(i)},
\end{equation}
\begin{equation}
\label{eq21}
\Omega_{n} = n^{-1} \sum_{i=1}^{n} V_{g_i}[\nabla V_i (Y; {\bm{\theta}})].
\end{equation}
The matrix defined in Equation (\ref{eq20}) has the expression
\begin{equation}
\label{eq22}
\begin{aligned}
J^{(i)} &= \int u_{i}(y, {\bm{\theta}}_g) u_{i}^T(y, {\bm{\theta}}_g) w[f_{i}(y; {\bm{\theta}})] f_{i}(y; {\bm{\theta}}) dy \\
        &\quad + \int [ -\nabla u_{i}(y, {\bm{\theta}}_g)  -  u_{i}(y, {\bm{\theta}}_g)u_{i}^T(y, {\bm{\theta}}_g)  h_i(x)] \times \\
        &\quad \quad \quad \quad (g_i(x) - f_{i}(y; {\bm{\theta}}))  w[f_{i}(y; {\bm{\theta}})] dy
\end{aligned}
\end{equation}
where $w(t) = B''(t) \times t$, $w'(t) = B''(t) + B'''(t) \cdot t$ and $$h_i(t) = \frac{w'(f_{i}(t; {\bm{\theta}}))f_{i}(t; {\bm{\theta}})}{w(f_{i}(t; {\bm{\theta}}))}.$$ Similarly, the matrix defined in Equation (\ref{eq21}), has the expression
\begin{equation}
\label{eq23}
\begin{aligned}
\Omega_{n} &= \frac{1}{n} \sum_{i=1}^{n} \Big[ \int u_{i}(y, {\bm{\theta}}_g) u_{i}^T(y, {\bm{\theta}}_g) w^2[f_{i}(y; {\bm{\theta}})] dG_i(y) - \xi_i \xi_i^T \Big]   
\end{aligned}
\end{equation}
where $\xi_i  = \int u_{i}(y, {\bm{\theta}}_g)w[f_{i}(y; {\bm{\theta}})]  dG_i(y).$ As in the case for i.i.d. data, we will make the following assumptions to establish the asymptotic properties of the minimum EWD estimators. These are analogous to the assumptions given in \cite{ghosh2013}; appropriate generalizations have been made to serve the entire Bregman divergence family. 
\begin{enumerate}
  \item[(B1)] The support $\chi = \{y \mid f_i(y; {\bm{\theta}}) > 0\}$ is independent of $i$ and ${\bm{\theta}}$ for all $i$; the
true distributions $G_i$ are also supported on $\chi$ for all $i$.
    \item[(B2)] There is an open subset $\omega$ of the parameter space $\Omega$, containing the
best fitting parameter ${\bm{\theta}}_g$ such that for almost all $y \in \chi$, and all ${\bm{\theta}} \in \Omega$,
all $i = 1, 2, \ldots$, the density $f_i(y; {\bm{\theta}})$ is thrice differentiable with respect to
${\bm{\theta}}$ and the third partial derivatives are continuous with respect to ${\bm{\theta}}$.
\item[(B3)] For $i = 1, 2, \ldots$, the integrals $\int [f_{i}(y; {\bm{\theta}}) B'(f_{i}(y; {\bm{\theta}})) - B(f_{i}(y; {\bm{\theta}})) ] dy$ and  $\int [B'(f_{i}(y; {\bm{\theta}})) g_{i}(y) ] dy$ can be differentiated thrice with respect to ${\bm{\theta}}$, and the derivatives can be taken under the integral sign.
\item[(B4)] For each $i = 1, 2, \ldots$, the matrices $J^{(i)}$ are positive definite and
\begin{equation*}
  \lambda_0 = \underset{n}{\inf}[\text{min eigenvalue of}\ \Psi_n] > 0
\end{equation*}
\item[(B5)] There exists a function $M^{(i)}_{jkl}(y)$ such that
\begin{equation*}
\begin{aligned}
&\mid \nabla_{jkl} V_{i}(y; {\bm{\theta}})\mid \leq M_{jkl}^{(i)}(y) \quad \forall \ {\bm{\theta}} \in \Omega, \quad \forall i = 1, 2, \ldots \\
\text{where}\quad &\frac{1}{n}\sum_{i=1}^{n} E_{q_i}\big[ M_{jkl}^{(i)}(Y) \big] = O(1) \quad \forall j, k, l.
\end{aligned}
\end{equation*}
\item[(B6)] For all $j, k$ we have
\begin{equation*}
\begin{aligned}
\underset{N \rightarrow \infty}{\lim} \underset{n}{\sup} \quad \Big\{ \frac{1}{n} \sum_{i=1}^{n} E_{g_i}\big[ \mid \nabla_{j} V_{i}(Y; {\bm{\theta}})  \mid  I (\mid \nabla_{j} V_{i}(Y; {\bm{\theta}})  \mid > N)\big]\Big\} = 0.
\end{aligned}
\end{equation*}

\begin{equation*}
\begin{aligned}
\underset{N \rightarrow \infty}{\lim} \underset{n}{\sup} \quad \Big\{ \frac{1}{n} \sum_{i=1}^{n}
E_{g_i}\Big[ & \big( \mid \nabla_{jk} V_{i}(Y; {\bm{\theta}})  - E_{g_i}( \nabla_{jk} V_{i}(Y; {\bm{\theta}})) \mid \big)  \\
& \times I \big( \mid \nabla_{jk} V_{i}(Y; {\bm{\theta}})  - E_{g_i}( \nabla_{jk} V_{i}(Y; {\bm{\theta}})) \mid \   > N \big) \Big]\Big\} = 0.
\end{aligned}
\end{equation*}

\item[(B7)] For all $\epsilon > 0$, we have
\begin{equation*}
\underset{N \rightarrow \infty}{\lim} \Big\{ \sum_{i=1}^{n}
E_{g_i} \Big[   \mid\mid   \Omega_{n}^{-1/2} \nabla V_{i}(Y; {\bm{\theta}}) \mid\mid^2 \  I (\mid\mid \Omega_{n}^{-1/2} \nabla V_{i}(Y; {\bm{\theta}}) \mid\mid  \ > \ \epsilon \sqrt(n)\Big]\Big\} = 0.
\end{equation*}
\end{enumerate}

\begin{thm}
If assumptions (B1)–(B7) hold, the following results are true.
 \begin{enumerate}
  \item There exists a consistent sequence ${\bm{\theta}}_n$ of roots satisfying the minimum Bregman divergence estimating equation given by Equation (\ref{eq17}).
  \item The asymptotic distribution of $\Omega_{n}^{-1/2}\Psi_n [\sqrt{n}({\bm{\theta}}_n - {\bm{\theta}}_g)]$ is s-dimensional normal with (vector) mean 0 and covariance matrix $I_s$, the s-dimensional identity matrix.
\end{enumerate}
\label{thm:thm2}
\end{thm}
\begin{proof}
The proof of this theorem follows exactly like the proof presented in Appendix 1 of \cite{ghosh2013}. 
\end{proof}

\begin{remark}
{\em On assumptions (B1) - (B7).} The assumptions (B1)–(B5) are simple generalizations of the assumptions (A1)-(A5) presented in Section \ref{subsec:asymp} of this manuscript. The assumptions (B6) and (B7) are similar in spirit to the corresponding assumptions required in the case of the maximum likelihood estimators under the similar independent non-homogeneous set-up as discussed in \cite{ibragimov1981}. These assumptions hold automatically for minimum Bregman divergence estimators in the i.i.d. case (see remark below). In subsequent sections we will see that these assumptions hold, for example, for the normal linear regression models under some mild conditions on the regressor variables.
\end{remark}

\begin{remark}
{\em For homogeneous data: a special case.} Note that, setting $f_i = f$ for all $i$, we get back the
corresponding asymptotic properties of the minimum Bregman divergence estimator for the i.i.d. case as given in Section \ref{subsec:asymp}. If $f_i = f,\quad i = 1, 2, \ldots$, we get $J^{(i)} = J,\ \xi_i = \xi$ for all i; thus $\Psi_n = J$ and $\Omega_n = K$. Here
$J, K\ \text{and}\ \xi$ are as defined in Section \ref{subsec:asymp}. In this case assumptions (B1)–(B5) are exactly the same as the assumptions (A1)-(A5) given in Section \ref{subsec:asymp}, while assumptions (B6) and (B7) are automatically satisfied by the  dominated convergence theorem. Thus, Theorem \ref{thm:thm1}, which establishes the consistency and asymptotic normality of the minimum Bregman divergence based estimator $\hat{{\bm{\theta}}}$ with $\sqrt{n}(\hat{{\bm{\theta}}_n} - {\bm{\theta}}_g)$ having the asymptotic covariance matrix $\Psi_n^{-1} \Omega_n \Psi_n^{-1} = J^{-1}KJ^{-1}$,  emerges as a special case of Theorem \ref{thm:thm2}. 
\end{remark}

\subsection{Application: Normal linear regression}
\label{linReg}
In this section, we will see that the theory proposed in Section \ref{inhd:asymp} would be immediately applicable in the case of linear regression under some mild conditions on the regressor variables. Specifically, the methodology described previously will immediately fall into place for the case of linear regression set-up with normal errors where the conditional approach to inference given fixed values of the explanatory variable is adopted. In this section, we will discuss applications of the proposed method in case of linear regression. Consider the linear regression model
\begin{equation*}
y_i = x^{T}_i \bm{\gamma} + \epsilon_i, \quad i= 1, 2, \ldots
\end{equation*}
where the error $\epsilon_i$’s are i.i.d. normal variables with mean zero and variance $\sigma^2$, $\bm{x}^T_i = (x_{i1}, x_{i2}, \ldots, x_{is})$ is the vector of the independent variables corresponding to
the $i$-th observation and $\bm{\gamma} = (\gamma_1,\ldots, \gamma_s)^T$ represents the regression coefficients.
We will assume that $\bm{x}_i$’s are fixed. Then $y_i \sim N(\bm{x}^T_i \bm{\gamma}, \sigma^2)$ and hence the $y_i$’s are independent but not identically distributed. Thus $y_i$’s satisfy the set-up of Sections \ref{inhd:intro} and \ref{inhd:asymp} and hence the minimum Bregman divergence estimator of the parameter ${\bm{\theta}} = (\bm{\gamma}^T, \sigma)^T$ can be obtained by minimizing the expression in Equation (\ref{eq16}) with $$f_i(y; {\bm{\theta}}) = \frac{1}{\sigma}\phi\Big(\frac{y - \bm{x}_i^T \bm{\gamma}}{\sigma}\Big)$$ where $\phi(\cdot)$ is the pdf of a standard normal random variable. Following the notation of Equation  (\ref{eq17}), we have the score equation as
\begin{equation*}
 \sum_{i=1}^{n} \Big[ u_i(Y_i; {\bm{\theta}})w[f_{i}(Y_i; {\bm{\theta}})] - \int u_i(t; {\bm{\theta}})w[f_{i}(t; {\bm{\theta}})]f_{i}(t; {\bm{\theta}})dt \Big] = 0,
\end{equation*}
where $w(t) = B''(t) \times t$ and the score function is given by
\begin{equation*}
  u_i(Y_i; {\bm{\theta}}) = \begin{bmatrix}
  \frac{(Y_i - \bm{x}_i^T\bm{\gamma)}}{\sigma^2} \bm{x}^T_i \quad; \quad
  \frac{(Y_i - \bm{x}_i^T\bm{\gamma})^2 - \sigma^2}{\sigma^3}
\end{bmatrix}^T.
\end{equation*}
Thus, we get the set of $s+1$ estimating equations:  
\begin{align*}
     \sum_{i=1}^{n} & \ x_{ij} (y_i - \bm{x}_i^T\bm{\gamma}))w(f_i(y; {\bm{\theta}})) = 0 \quad \forall \ j=1, \ldots, s, \\
   n^{-1}\sum_{i=1}^{n} & \Big[ \Big[\frac{(y - x_i^T\gamma)^2 - \sigma^2}{\sigma^3} \Big] w(f_i(y; {\bm{\theta}}))  - \\ & \quad \quad \int \Big[\frac{(y - x_i^T\gamma)^2 -  \sigma^2}{\sigma^3}\Big] w(f_i(y; {\bm{\theta}})) f_i(y; {\bm{\theta}}) dy \Big] = 0.
\end{align*}
Now, we can then solve these $s+1$ estimating equations numerically to obtain the estimates of ${\bm{\theta}}.$ 

\begin{remark}
{\em On asymptotic behaviour of the estimator $\hat{{\bm{\theta}}} = (\hat{\bm{\gamma}}^T, \hat{\sigma})^T
$}: For simplicity we will assume that the true data generating density $g_i$ also belongs to the model family of distributions, i.e., $g_i(\cdot) = f_{i}(\cdot; {\bm{\theta}})\ \forall\ i=1, 2, \ldots$ Then we can derive the simplified form of the matrices $\Psi_n\ \text{and}\ \Omega_n.$
We had previously defined $\Psi_n = n^{-1} \sum_{i=1}^{n} J^{(i)}$ and $\Omega_n = n^{-1} \sum_{i=1}^{n} K^{(i)}$. Using Equations (\ref{eq22}) and (\ref{eq23}), and the fact that $g_i(\cdot) = f_{i}(\cdot; {\bm{\theta}})\ \forall\ i=1, 2, \ldots$, we have
\begin{equation*}
\begin{aligned}
J^{(i)} &= \int u_{i}(y, {\bm{\theta}}_g) u_{i}^T(y, {\bm{\theta}}_g) w[f_{i}(y; {\bm{\theta}})] f_{i}(y; {\bm{\theta}}) dy,\\
K^{(i)} &= \int u_{i}(y, {\bm{\theta}}_g) u_{i}^T(y, {\bm{\theta}}_g) w^2[f_{i}(y; {\bm{\theta}})] f_{i}(y; {\bm{\theta}}) dy - \xi^{(i)} {\xi^{(i)}}^T,\\
\xi^{(i)}  &= \int u_{i}(y, {\bm{\theta}}_g)w[f_{i}(y; {\bm{\theta}})] f_{i}(y; {\bm{\theta}}) dy.
\end{aligned}
\end{equation*}
As in the previous sections, we have $w(t) = B''(t) \times t$;  in the case of EWD($\beta$), $w(t) = 1 - \exp(-t/\beta)$. 
It can be shown that $\hat{{\bm{\theta}}}_n$ is a consistent estimator of ${\bm{\theta}}$. Further, the asymptotic distribution of $\sqrt{n} \Omega_n^{-1/2}\Psi_n(\hat{{\bm{\theta}}}_n -  {\bm{\theta}})$ is multivariate normal with mean (vector) zero and covariance matrix $I_{s+1}$. This can be proved by consulting Theorem \ref{thm:thm2}.
\end{remark}

In the next section, we will see how this method works in the context of some real life data sets.

\subsubsection{Simple linear regression: Homicide from firearms and GDP}
\label{gun}

As an application of the robust regression method developed in Section \ref{linReg}, we consider modeling age-standardized national firearm-related homicide rates in 23 Western countries as a function of per-capita gross domestic product as of 2017. Information on GDP was obtained from \cite{cia} and data on firearm-related homicide rates were obtained from \cite{owidhomicides}. Figure \ref{fig:f9} is a scatter-plot of the data set described, where the independent variable per-capita gross domestic product is plotted on the X-axis, and firearm related homicide rate on the Y-axis. The United States of America has an abnormally high firearm related homicide rate in relation to its per-capita GDP, and this single outlier forces the least squares regression line to have a positive slope, which clearly contradicts the general configuration of points.

In comparison, the two MEWDE fits show a clear reversal in slope, and give more satisfactory descriptions of the rest of the data, sacrificing the large outlier. Table \ref{tab:tab7} shows how estimated coefficients vary as we change the tuning parameter $\beta$. We observe that for very small $\beta = 0.002$, our MEWDE estimator almost mimics the MLE$+$D estimator, implying that the MEWDE fits the data well by automatically downweighting the outlier, even for very small values of $\beta$.

\begin{table}[h]
\centering
\caption{Estimated regression parameters for homicide data.}
\label{tab:tab7}
\begin{tabular}{lccccccc}
\toprule
Estimates & MLE & E(0.002) & E(0.02) &  E(0.25) &  E(1) &  MLE+D  \\ \midrule
Intercept       & $-0.293$ & $0.356$ & $0.356$ & $0.404$ & $0.359$ & $0.356$ \\ 
GDP ($\times 10^{-6}$)         & $14.49$  &  $-3.045$   & $-3.042$&  $-4.087$  &  $-3.250$   &   $-3.042$   \\
\midrule
Error s.d.  & $0.959$ & $0.110$ & $0.111$ & $0.131$ & $0.106$ & $0.088$ \\ 
\bottomrule
\end{tabular}
\end{table}       

\begin{figure}[h]
  \centering
  \captionsetup{justification=centering}
  \includegraphics[width=0.8\textwidth]{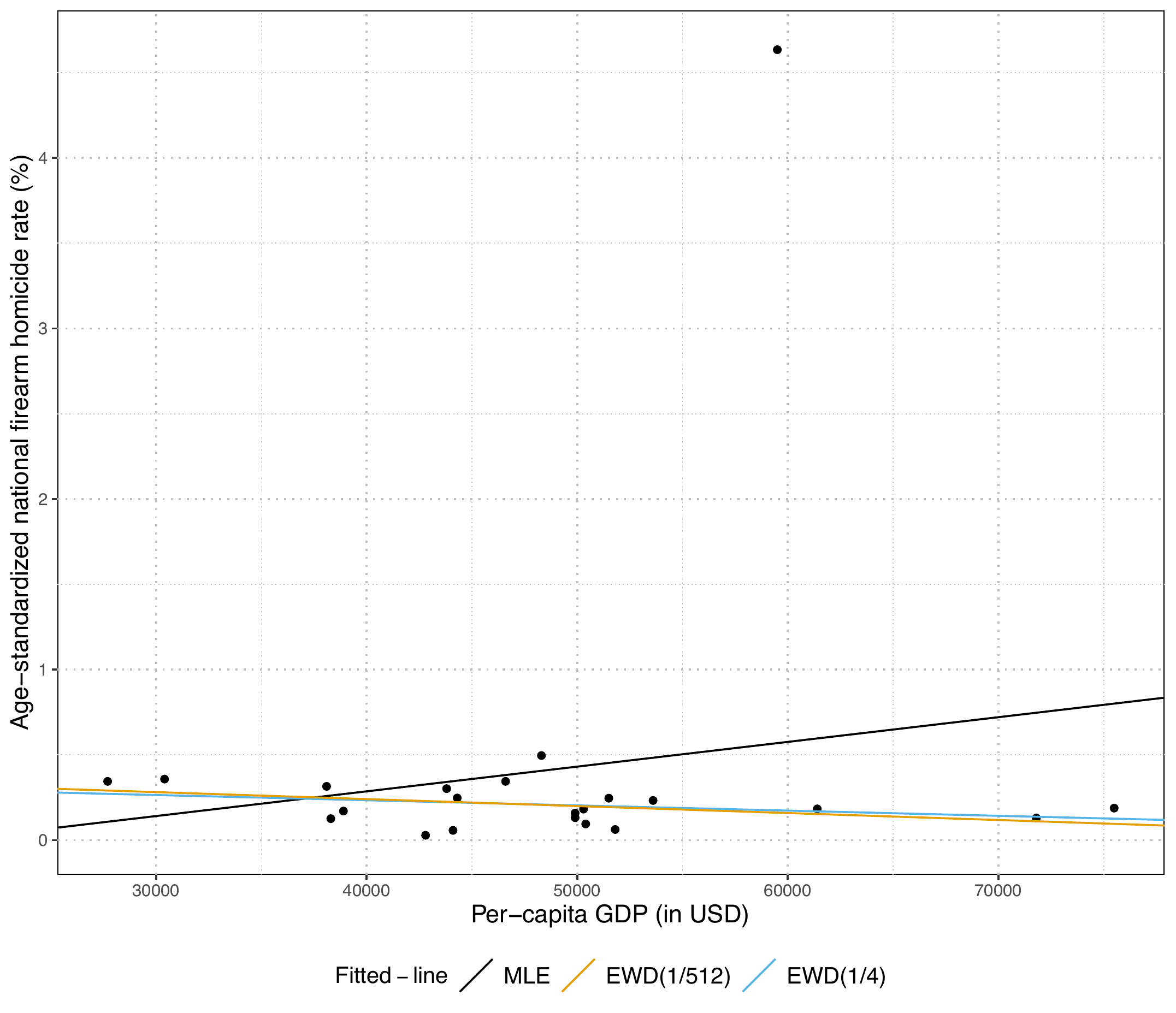}
  \caption{Modeling firearm-related homicide rates in  Western  countries  as  a  function  of  per-capita  gross  domestic  product: fits with ML and minimum EWD estimators.}\label{fig:f9}
\end{figure}

\subsubsection{Other examples of simple linear regression}
We have analyzed two other datasets, the Belgian telephone data and the Hertzsprung Russell star cluster data, both available in \cite{rousseeuw1987}, using the minimum divergence procedures based on the EWD and DPD. The details are provided in Appendix \ref{appBT}.

\subsubsection{Multiple linear regression: Alcohol solubility data}
\label{alcohol}
We consider fitting a multiple linear regression model to the dataset concerning alcohol solubility in water \citep{maronna2019}. The dataset gives, for 44 aliphatic alcohols, the logarithm of their solubility together with three physicochemical characteristics (namely, solvent accessible surface-bounded molecular volume (SAG), 
mass and volume). The interest is in predicting the solubility. Following the authors' suggestion of fitting an MM regression-based model to the data, we observe that 
four data points (roughly $10\%$ of the data set) are assigned much smaller `robustness weights' as compared to the remaining 40 data points.  Treating these four observations as outliers, we obtain the outlier-deleted maximum likelihood estimates (denoted by MLE$+$D) of the regression coefficients and error standard deviation. We also compute the robust LMS estimate. In order to estimate the error s.d. $\sigma$, we compute $\hat{\sigma} = \text{median} \lvert r_i - \text{median}(r_i) \rvert / 0.67449$.
\begin{table}[h]
\centering
\caption{Estimated regression parameters for alcohol solubility data (\cite{maronna2019}).}
\label{tab:tab8}
\begin{tabular}{lcccccccc}
\toprule
Estimates & MLE & LMS & E(0.1) &  E(0.4) & E(0.7) & MLE$+$D \\ \midrule
Intercept & $8.777$ & $3.617$ & $5.883$ & $3.974$ & $5.444$ & $6.829$ \\
SAG & $0.014$ & $0.177$ & $0.110$ & $0.163$ & $0.129$ & $0.077$ \\ 
Volume & $-0.040$ & $-0.191$ & $-0.133$ & $-0.179$ & $-0.152$ & $-0.102$ \\ 
Mass & $0.027$ & $0.248$ & $0.172$ & $0.235$ & $0.206$ & $0.127$ \\  \midrule
Error s.d. & $0.504$ & $0.405$ & $0.372$ & $0.145$ & $0.221$ & $0.389$ \\
\bottomrule
\end{tabular}
\end{table}
Finally, we compute minimum EWD($\beta$) regression parameter estimates for various values of $\beta$. Our findings are presented in Table \ref{tab:tab8}. 

Unlike simple linear regression, where the fit can be plotted and its suitability visually examined,  a visual inspection is not possible for the fits this multiple linear regression model. Thus the coefficients of Table \ref{tab:tab8} are not alone sufficient to give a full idea about how good the fits are, how stable and outlier-resistant their behaviors are. We therefore look at the residuals of each of these fits and try to determine how well they fare in terms of separating out the outliers. When there is a small number of outliers in the data, a robust and outlier-resistant procedure is likely to fit the good data part adequately and make the outliers stand out in terms of fitted residuals. A robust and resistant fit is supposed to properly model the majority good data and sacrifice the stray outliers, which then stand out  in terms of residuals.
\begin{figure}[h]
  \centering
  \captionsetup{justification=centering}
  \includegraphics[width=0.8\textwidth]{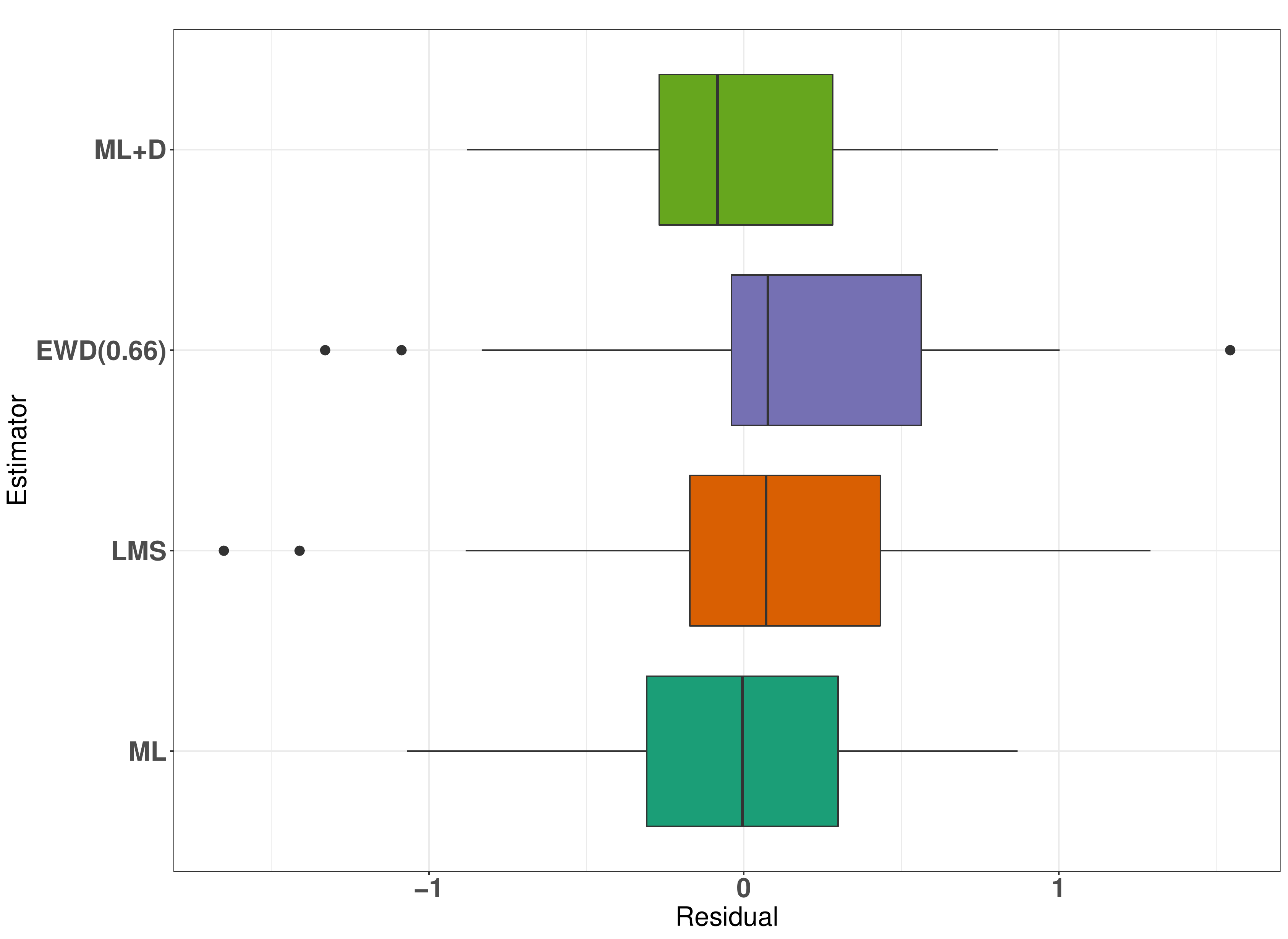}
  \caption{Residual boxplots of ML, LMS, ML+D and minimum EWD($0.66$) fits for alcohol solubility data \citep{maronna2019}.}\label{fig:f11}
\end{figure}
The non-robust fits, on the other hand, are highly affected by the outliers, and the residuals of this fit may no longer stand out, but get masked with the other residuals. In Figure \ref{fig:f11} we present the boxplots of the residuals of the ML (LS) fit, ML+D (outlier deleted LS) fit, the LMS fit and the minimum EWD(0.66) fit. The optimal tuning parameter is found to be $\beta_{OPT} = 0.66$. See Section \ref{tp:inhd} for more details.

It may be seen that the LMS and minimum EWD(0.66) procedure identifies two and three outliers, respectively, by the basic boxplot method. On the other hand, for the ML method the residuals of the outliers are masked with the good data, while for the ML+D method there are no outliers.  We also present the residual plots (against fitted values) of these four fits as well as the kernel density estimates of these outliers in Appendix \ref{app_alcohol} for further substantiation of this description.  

\subsection{Tuning parameter selection}
\label{tp:inhd}
As an extension of Section \ref{ihd:tune}, we refer to \cite{ghosh2013}, where the problem of tuning parameter selection in the context of independent and non-homogeneous data was discussed. The generalizations required in the case of minimum Bregman divergence estimation are relatively straightforward, so we do not dwell on those here. 

For the data on firearm-related homicide and GDP presented in Section \ref{gun}, we obtain the optimal tuning parameter to be $\beta_{OPT} = 1.6$, and the corresponding estimated regression parameters (intercept, GDP, error standard deviation) are $(0.416, -3.932 \times 10^{-6}, 0.066)$.

Similarly, for the data on alcohol solubility presented in Section \ref{alcohol}, we obtain the optimal tuning parameter to be $\beta_{OPT} = 0.66$, and the corresponding estimated regression parameters (intercept, SAG, Volume, Mass, error standard deviation) are $(6.084, 0.112, -0.135, 0.174, 0.062)$.

\section{Testing of hypotheses}
\subsection{Introduction}
In the following subsections, we make use of the EWD in constructing robust tests of hypotheses based on the Bregman divergence and the corresponding minimum divergence estimators.  Our work may be viewed as a generalization of the work presented in  \cite{basu2013testing} and \cite{basu2018testing}. We establish the asymptotic null distribution of the proposed test statistic and apply the theory developed to a real-life data set. As in the previous sections, our focus will remain on the exponentially weighted divergence. 

\subsection{Formulating the test statistic}
We begin with $\{ P_{\bm{{\bm{\theta}}}}: \bm{{\bm{\theta}}} \in \Omega \}$, an identifiable parametric family of probability measures on a measurable space $\{\chi, \mathscr{A} \}$ 
with an open parameter space $\Omega \subset \mathbb{R}^p,\ p\geq1.$ Measures $P_{\theta}$ are described by densities $f_{\theta}  = dP_{\theta} /d\mu$, absolutely continuous with 
respect to a dominating $\sigma$-finite measure $\mu$ on $\chi$. We have a sample of size $n$ given by $X_1, X_2, \ldots, X_n$ from a 
density belonging to the family $\mathscr{F}_{\bm{{\bm{\theta}}} } = \{f_{\bm{{\bm{\theta}}}} :  \bm{{\bm{\theta}}}  \in \Omega \}$. 
We will assume that the support of the distribution is independent of $ \bm{{\bm{\theta}}} $. Our aim is to test a general null hypothesis of the form 
\begin{equation}
\label{eq24}
H_0 : \bm{{\bm{\theta}}} \in \Omega_0 \text{ against } H_1: \bm{{\bm{\theta}}} \notin \Omega_0.    
\end{equation}
As in many practical hypothesis testing problems, we consider the set-up where the restricted parameter space specified by $H_0$ can be rewritten by a set of $r < p$ restrictions of the form 
\begin{equation}
\label{eq25}
\bm{m({\bm{\theta}}) = 0}_r
\end{equation}
on $\Omega$, where $\bm{m}: \mathbb{R}^p \rightarrow \mathbb{R}^r$ is a vector valued function such that the $p \times r$ matrix $$M({\bm{\theta}}) = \frac{\partial \bm{g^T({\bm{\theta}})}}{\partial \bm{{\bm{\theta}}}}$$ exists and is continuous in $\bm{{\bm{\theta}}}$ and rank($M({\bm{\theta}})$) $= r$. 

Given a sample, our approach to solving the hypothesis testing problem described in Equation (\ref{eq24}) will be to first obtain $\bm{\hat{{\bm{\theta}}}}_{B_1}$, the unrestricted minimum Bregman divergence estimator for a given $B_1$ function, then obtain the restricted minimum Bregman divergence estimator $\bm{\tilde{{\bm{\theta}}}}_{B_1}$, subject to the constraints specified by Equation (\ref{eq25}), for the same $B_1$ function. Finally, we will look at the family of Bregman divergence test statistics (BDTS) 
\begin{equation}
\label{eq26}
    T_{B_2}(\hat{\bm{{\bm{\theta}}}}_{B_1}, \tilde{\bm{{\bm{\theta}}}}_{B_1} ) = 2n \times {D}_{B_2}(f_{\bm{\hat{{\bm{\theta}}}}_{B_1}}, f_{\bm{\tilde{{\bm{\theta}}}}_{B_1}}),
\end{equation} where $D_{B_2} (g, f)$ is the Bregman divergence between two densities $g$ and $f$, defined in Equation (\ref{eq1}) with $B_2$ as the $B$ function. The asymptotic distribution of the test statistic can be worked out for the case where the functions $B_1$ and $B_2$ are distinct, and, for maximum flexibility of the method, we establish the asymptotic results of our testing procedure for this general case. In practice, however, it is not easy to determine the benefits of having two different functions in these two roles, and a single, suitably chosen function will generally work well in most cases. We will consider the functions $B_1$ and $B_2$ to have the same parametric form,  corresponding to exponentially weighted divergences, only differing, if at all, in the values of their tuning parameters. 

In Section \ref{subsec:asymp}, we have established the asympotic behaviour of the unrestricted minimum Bregman divergence estimator. In order to establish the asymptotic behaviour of the Bregman divergence test statistic, we first obtain the asymptotic distribution of restricted minimum Bregman divergence estimator $\bm{\tilde{{\bm{\theta}}}}_{B_1}$, and then work out the asymptotic properties of the family of test statistics given by Equation (\ref{eq26}).

\subsection{Restricted minimum Bregman divergence estimator}

\begin{thm} In addition to assumptions (A1) to (A5) in Section \ref{subsec:asymp}, we make the assumption 
\label{thm:rest}
\begin{enumerate}
    \item[(A6)] For all $\bm{{\bm{\theta}}} \in \omega$, the partial derivatives $\partial^2 m_l(\bm{{\bm{\theta}}})/ \partial {\bm{\theta}}_j \partial {\bm{\theta}}_k$ are bounded for all $j$, $k$ and $l$, where $m_l(\cdot)$ is the $l$-th element of $\bm{m}(\cdot)$
\end{enumerate}
We also assume that the true distribution belongs to the model and $\bm{{\bm{\theta}}}_0 \in  \Omega_0$ is the true parameter. Under this set-up, the minimum Bregman divergence estimator $\tilde{\bm{{\bm{\theta}}}}_{B_1}$ obtained under the constraints $\bm{m({\bm{\theta}}) = 0}_r$  has the following asymptotic properties. 
\begin{enumerate}
    \item The restricted minimum Bregman divergence estimating equation has a consistent sequence of roots, i.e., $$\tilde{\bm{{\bm{\theta}}}}_{B_1} \underset{n \rightarrow \infty}{\overset{P}{\rightarrow}} \bm{{\bm{\theta}}}_0.$$
    \item The null distribution of $\sqrt{n} (\tilde{\bm{{\bm{\theta}}}}_{B_1} - \bm{{\bm{\theta}}}_0)$ is given by an $p$ dimensional multivariate normal distribution with the zero mean vector and an $p\times p$ dispersion matrix $\Sigma_{B_1}(\bm{{\bm{\theta}}}_0)$. This matrix is defined as 
    \begin{equation}
        \label{eq27}
        \Sigma_{B_1}(\bm{{\bm{\theta}}}) = P_{B_1}(\bm{{\bm{\theta}}}) K_{B_1}(\bm{{\bm{\theta}}}) P_{B_1}(\bm{{\bm{\theta}}}),
    \end{equation}
    where $K_{B_1}(\bm{{\bm{\theta}}})$ is defined by Equation (\ref{eq11}) with $B_1$ as the relevant $B$ function. The matrix $P_{B_1}(\bm{{\bm{\theta}}})$ is defined as 
    \begin{equation}
        \label{eq28}
        P_{B_1}(\bm{{\bm{\theta}}}) = J^{-1}_{B_1}(\bm{{\bm{\theta}}}) - Q_{B_1}(\bm{{\bm{\theta}}}) M^T_{B_1}(\bm{{\bm{\theta}}}) J^{-1}_{B_1}(\bm{{\bm{\theta}}}), 
    \end{equation}
    where $J_{B_1}(\bm{{\bm{\theta}}})$ is defined by Equation (\ref{eq12}) with $B_1$ as the associated $B$ function, and the matrix $Q_{B_1}(\bm{{\bm{\theta}}})$ is defined as
    \begin{equation}
        \label{eq29}
        Q_{B_1}(\bm{{\bm{\theta}}}) = J^{-1}_{B_1}(\bm{{\bm{\theta}}}) M_{B_1}(\bm{{\bm{\theta}}})  \Big[M^T_{B_1}(\bm{{\bm{\theta}}}) J^{-1}_{B_1}(\bm{{\bm{\theta}}})  M_{B_1}(\bm{{\bm{\theta}}})\Big]^{-1}. 
    \end{equation}
\end{enumerate}
\end{thm}

\begin{proof}
The proof of this theorem follows exactly like the proof presented in the Appendix  of \cite{basu2018testing}. 
\end{proof}

It is interesting to note that Theorem \ref{thm:rest} is an extension of Theorem \ref{thm:thm1}. While the former allows for estimation in a restricted parameter space, the latter does not. As a result, when dealing with an unrestricted parameter space, ${M}$ becomes a null matrix and consequently, $P_{B_1}(\bm{{\bm{\theta}}}) = J^{-1}_{B_1}(\bm{{\bm{\theta}}})$ and the asymptotic dispersion matrix of the unrestricted minimum Bregman divergence estimator reduces to the form specified by Theorem \ref{thm:thm1}.

\subsection{Bregman divergence test statistic}
First, we fix a function $B_1$ and denote $\hat{\bm{{\bm{\theta}}}}_{B_1}$ as the unconstrained minimum Bregman divergence estimator of $\bm{{\bm{\theta}}}$, and $\tilde{\bm{{\bm{\theta}}}}_{B_1}$ as the restricted estimator under the null hypothesis specified by Equation (\ref{eq24}). Next, we consider another function $B_2$ (which is, for simplicity, assumed to have the same functional form as $B_1$, only differing in the value(s) of tuning parameter(s)) and construct the BDTS, as defined in Equation (\ref{eq26}). In the following theorem, we present the asymptotic distribution of the family of BDTS. 
\begin{thm}
\label{thm:bdts}
We assume that conditions (A1) - (A5) of Theorem \ref{thm:thm1} and condition (A6) of Theorem \ref{thm:rest} holds. The asymptotic distribution of $T_{B_2}(\hat{\bm{{\bm{\theta}}}}_{B_1}, \tilde{\bm{{\bm{\theta}}}}_{B_1} )$ defined in Equation (\ref{eq26}) coincides with, under the null hypothesis specified in Equation (\ref{eq24}), the distribution of the random variable $$\sum_{i=1}^{k} \lambda_{i}(B_1, B_2, \bm{{\bm{\theta}}}) Z^2_i,$$
where $Z_1, \ldots, Z_k$ are independent standard normal variables and $\lambda_i(B_1, B_2, \bm{{\bm{\theta}}})$ for $i= 1, \ldots, k$ are the nonzero eigenvalues of the matrix 
$$A_{B_2}(\bm{{\bm{\theta}}_0}) B_{B_1}(\bm{{\bm{\theta}}_0}) K_{B_1}( \bm{{\bm{\theta}}_0}) B_{B_1}(\bm{{\bm{\theta}}_0}),$$ and $k$ is the rank of the matrix given by
\begin{equation} 
\label{eq29}
     B_{B_1}(\bm{{\bm{\theta}}_0}) K_{B_1}( \bm{{\bm{\theta}}_0}) B_{B_1}(\bm{{\bm{\theta}}_0})
                    A_{B_2}(\bm{{\bm{\theta}}_0}) B_{B_1}(\bm{{\bm{\theta}}_0}) K_{B_1}( \bm{{\bm{\theta}}_0}) B_{B_1}(\bm{{\bm{\theta}}_0}).
\end{equation}
The matrix $A_{B_2}(\bm{{\bm{\theta}}_0})$ is defined element-wise by
\begin{equation}
    \label{eq30}
    \begin{split}
     & = \Big(a^{B_2}_{ij} (\bm{{\bm{\theta}}_0}) \Big)_{p \times p} \\
                        & = \int \Bigg[ B_2''(f_{\bm{{\bm{\theta}}_0}}(x)) \frac{\partial f_{\bm{{\bm{\theta}}_0}}(x)}{\partial {\bm{\theta}}_i} \frac{\partial f_{\bm{{\bm{\theta}}_0}}(x)}{\partial {\bm{\theta}}_j} \Bigg] dx,
\end{split}
\end{equation} and $B_{B_1}(\bm{{\bm{\theta}}_0})$ is the matrix
\begin{equation}
   \label{eq31} 
      \boldsymbol{J}_{B_1}^{-1}\left(\boldsymbol{{\bm{\theta}}}_{0}\right) \boldsymbol{M}\left(\boldsymbol{{\bm{\theta}}}_{0}\right)\left[\boldsymbol{M}^{T}\left(\boldsymbol{{\bm{\theta}}}_{0}\right) \boldsymbol{J}_{B_1}^{-1}\left(\boldsymbol{{\bm{\theta}}}_{0}\right) \boldsymbol{M}\left(\boldsymbol{{\bm{\theta}}}_{0}\right)\right]^{-1} \boldsymbol{M}^{T}\left(\boldsymbol{{\bm{\theta}}}_{0}\right) \boldsymbol{J}_{B_1}^{-1}\left(\boldsymbol{{\bm{\theta}}}_{0}\right).
\end{equation}
\end{thm}
\begin{proof}
The proof of this theorem resembles the proof of Theorem 6 of \cite{basu2018testing}. 
\end{proof}

\begin{remark}
We observe that the ranks of $$\boldsymbol{B}_{B_1}\left(\boldsymbol{{\bm{\theta}}}_{0}\right) \boldsymbol{K}_{B_1}\left(\boldsymbol{{\bm{\theta}}}_{0}\right) \boldsymbol{B}_{B_1}\left(\boldsymbol{{\bm{\theta}}}_{0}\right)$$
and  
$$\boldsymbol{B}_{B_1}\left(\boldsymbol{{\bm{\theta}}}_{0}\right) \boldsymbol{K}_{B_1}\left(\boldsymbol{{\bm{\theta}}}_{0}\right) \boldsymbol{B}_{B_1}\left(\boldsymbol{{\bm{\theta}}}_{0}\right) \boldsymbol{A}_{B_2}\left(\boldsymbol{{\bm{\theta}}}_{0}\right) \boldsymbol{B}_{B_1}\left(\boldsymbol{{\bm{\theta}}}_{0}\right) \boldsymbol{K}_{B_1}\left(\boldsymbol{{\bm{\theta}}}_{0}\right) \boldsymbol{B}_{B_1}\left(\boldsymbol{{\bm{\theta}}}_{0}\right)$$
are equal. Further, $\boldsymbol{B}_{B_1}\left(\boldsymbol{{\bm{\theta}}}_{0}\right) \boldsymbol{K}_{B_1}\left(\boldsymbol{{\bm{\theta}}}_{0}\right) \boldsymbol{B}_{B_1}\left(\boldsymbol{{\bm{\theta}}}_{0}\right)$ and $\boldsymbol{M({\bm{\theta}}_0)}$ both have the same rank $r$. So, $k=r$ and there are exactly $r$ many nonzero eigenvalues. 
\end{remark}

\begin{remark}
The critical region of the BDTS needs to be found so that the test can be carried out. An easy way to approximate the required critical region is outlined here. From Theorem \ref{thm:bdts} it is obvious that the $k$ eigenvalues described are functions of $\bm{{\bm{\theta}}}_0$. Under the null, they can be estimated in a consistent manner by plugging in $\tilde{\bm{{\bm{\theta}}}}_{B_1}$ in place of $\bm{{\bm{\theta}}}_0$. Let these estimated eigenvalues be $\hat{\lambda}_1, \cdots, \hat{\lambda}_k$. Generating $k$ many independent observations $Z_1, \cdots, Z_k$ from the $N(0,1)$ distribution repeatedly, one can obtain empirical estimates of the quantiles of $\sum_{i} \hat{\lambda}_i Z^2_i$ by replicating this procedure a sufficient number of times.
\end{remark}

\begin{remark}
An approximate form of the power function of the BDTS can be obtained by following the steps outlined  by Theorem 7 of \cite{basu2018testing}.
\end{remark}

\subsection{Normal case: constructing the exponentially weighted divergence test statistic}
We will now focus on a special case of the Bregman divergence test statistic --- the exponentially weighted divergence test statistic (EWDTS). 
Under the $N(\mu, \sigma^2)$ model, consider the problem of testing 
\begin{equation}
\label{eq32}
    H_0: \mu = \mu_0 \ \text{versus}\ H_1: \mu \neq \mu_0,
\end{equation}
where $\sigma$ is an unknown nuisance parameter. We note that the unrestricted parameter space is $\Omega = \{(\mu, \sigma)^T: \mu \in \mathbb{R},\ \sigma \in \mathbb{R}^{+} \}$ and that the restricted paramter space is $\Omega_0 = \{(\mu_0, \sigma)^T: \sigma \in \mathbb{R}^{+} \}$. We consider the restriction $m(\bm{{\bm{\theta}}}) = \mu - \mu_0$ where $\bm{{\bm{\theta}}} = (\mu, \sigma)^T$ so that the null hypothesis can be rewritten as $$H_0: m(\bm{{\bm{\theta}}}) = 0.$$ As a result of this formulation, we are now ready to discuss the testing problem in the framework of Theorems \ref{thm:rest} and \ref{thm:bdts}. We have already discussed the unrestricted minimum exponentially weighted divergence estimation of unknown mean and standard deviation parameters for the normal distribution in Section \ref{shoshoni}. We denote the unrestricted minimum EWD estimator obtained by $\hat{\bm{{\bm{\theta}}}}_{\beta}$, $\beta$ being the tuning parameter associated with the $B$ function of the EWD. Now, we turn our attention to the restricted minimum EWD estimation of $\bm{{\bm{\theta}}} = (\mu, \sigma)$ when subject to the restriction $\mu = \mu_0$. The estimate will be obtained by minimising the empirical estimate of the exponentially weighted divergence 
\begin{equation}
    \label{eq33}
    \tilde{\bm{{\bm{\theta}}}}_\beta = \underset{\bm{{\bm{\theta}}} \in \Omega_0}{\text{argmin}}  \int \Big[ B'(f_{\bm{{\bm{\theta}}}}(x)) f_{\bm{{\bm{\theta}}}}(x) - B(f_{\bm{{\bm{\theta}}}}(x))   \Big] dx - n^{-1} \sum_{i=1}^{n} B'(f_{\bm{{\bm{\theta}}}}(X_i)), 
\end{equation}
where the second derivative of $B$ obeys the relation $B''(t) \times t = 1 - \exp(-t/\beta)$. It should be noted that we use the same tuning parameter $\beta$ when looking for restricted as well as unrestricted estimates of $\bm{{\bm{\theta}}}$. For the testing problem described by Equation (\ref{eq32}), we take $f_{\bm{{\bm{\theta}}}}(x) = \phi ((x-\mu)/\sigma)/\sigma$ and obtain the restricted estimate $\tilde{\bm{{\bm{\theta}}}}_\beta = (\mu_0, \tilde{\sigma}_\beta)$, where
\begin{align*}
    \tilde{\sigma}_\beta = \underset{\sigma > 0}{\text{argmin}}  \int \Bigg[ B' \Bigg(\frac{1}{\sigma} \phi & \Big(\frac{x-\mu_0}{\sigma} \Big) \Bigg) \frac{1}{\sigma} \phi \Big(\frac{x-\mu_0}{\sigma}\Big)  - B \Bigg(\frac{1}{\sigma} \phi \Big(\frac{x-\mu_0}{\sigma} \Big) \Bigg)   \Bigg] dx \\
    & - \frac{1}{n} \sum_{i=1}^{n} B' \Bigg(\frac{1}{\sigma} \phi \Big(\frac{X_i-\mu_0}{\sigma} \Big) \Bigg). 
\end{align*}
Now, for some tuning parameter $\gamma$, we construct the EWDTS required to test the hypothesis specified by Equation (\ref{eq32}). The test statistic is given by 
\begin{equation}
    \label{eq34}
    T_{\gamma} ({\hat{\bm{{\bm{\theta}}}}_\beta}, {\tilde{\bm{{\bm{\theta}}}}_\beta}) = 2n \times D_{\gamma} (f_{\hat{\bm{{\bm{\theta}}}}_\beta}, f_{\tilde{\bm{{\bm{\theta}}}}_\beta})
\end{equation}
where $D_\gamma$ is the exponentially weighted divergence with tuning parameter $\gamma$. For notational convenience, here we index the divergence $D$ (in the subscript) by the tuning parameter $\gamma$ of the EWD, rather than by the corresponding convex function. 

We have already established that the null hypothesis can be well specified by one linear constraint on $\mu$. So, using Theorem \ref{thm:bdts}, we can claim that the asymptotic null distribution of the EWDTS may be characterized by  $\lambda(\bm{{\bm{\theta}}_0}, \beta, \gamma) Z^2$, where $Z \sim N(0, 1)$ and $\lambda(\bm{{\bm{\theta}}_0}, \beta, \gamma)$ is the nonzero eigenvalue of the matrix $A_{\gamma}({\bm{\theta}}_0) B_{\beta}({\bm{\theta}}_0) K_{\beta}({\bm{\theta}}_0) B_{\beta}({\bm{\theta}}_0)$. Since the value of $\bm{{\bm{\theta}}}_0$ is unknown, we can plug in $\tilde{\bm{{\bm{\theta}}}}_\beta$ in its place and claim 
\begin{equation}
    \label{eq35}
    \frac{T_{\gamma} ({\hat{\bm{{\bm{\theta}}}}_\beta}, {\tilde{\bm{{\bm{\theta}}}}_\beta})}{\lambda(\tilde{\bm{{\bm{\theta}}}}_\beta, \beta, \gamma)} \underset{n \rightarrow \infty}{\overset{L}{\rightarrow}} Z^2.
\end{equation}

As we have noted before, $J_{\beta}$ and $K_{\beta}$ do not have neat closed form expressions. This is true for $A_{\gamma}$ as well. As a result, instead of trying to find a more detailed form of $\lambda(\tilde{\bm{{\bm{\theta}}}}_\beta, \beta, \gamma)$, we will make use of numerical approximations. 

\subsubsection{Real data example: Shoshoni rectangles.} 
We have previously examined these data in Section \ref{shoshoni}. The focal point of our discussion there was the minimum EWD estimation of $\bm{{\bm{\theta}}} = (\mu, \sigma)^T$ when the data are assumed to come from a $N(\mu, \sigma^2)$ distribution, both location and scale parameters being unknown. \cite{hettmansperger2010} note that if we were to implement the $t$-test to test for the hypothesis 
\begin{equation}
    H_0: \mu = 0.618 \text{ versus } H_1: \mu \neq 0.618,
    \label{hypo}
\end{equation}
 we would get a $p$ value of 0.053, which is at the borderline of significance at 5\% level. On the other hand, the nonparametric one-sample sign test 
returns an entirely insignificant $p$ value of 0.823. It would be interesting to investigate the performance of the EWDTS in such a scenario. 

In Figure \ref{fig:10}, we present a graph of the $p$-values of the test statistic $T_{\beta}(\hat{\bm{{\bm{\theta}}}}_\beta, \tilde{\bm{{\bm{\theta}}}}_\beta)$ for testing the hypothesis in Equation (\ref{hypo}) over a set of values of $\beta$. For $\beta \rightarrow 0$, the EWDTS becomes equivalent to, asymptotically (in $n$), the ordinary likelihood ratio test under the null. This statistic returns a significant $p$-value of 0.045 for the hypothesis in Equation \ref{hypo}, but with increasing $\beta$, the significance turns to insignificance very fast. It may be noted that the $p$-value for the $t$-test statistic for the outlier deleted (the three large outliers removed) data is $0.329$, which conforms to the $p$-values corresponding to the EWDTS for moderately large positive values of $\beta$. 

\begin{figure}[h]
    \centering
    \includegraphics[width=\textwidth]{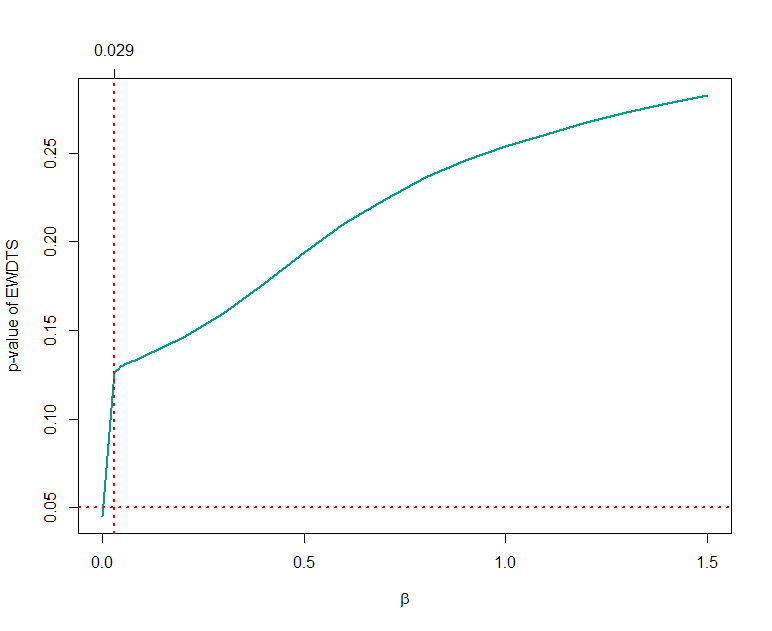}
    \caption{$p$-value of EWDTS for values of $\beta$ for Shoshoni rectangles data}
    \label{fig:10}
\end{figure}

\section{Concluding remarks}
In this paper, we have presented an estimator based on a sub-class of density-based Bregman divergences, which is seen to be outperforming the existing standard (i.e., the DPD based estimator). We have shown several asymptotic and distributional properties of the proposed estimator, both in the context of i.i.d data as well as independent and non-homogenous data. A special case of linear regression (both simple and multiple) has been explored in the context of real data. We have also discussed `judicial' choice(s) of the tuning parameter which, when chosen properly, yields highly robust and efficient estimators which can often dominate the MDPDE. We have also considered an hypothesis testing strategy for parameteric models which may serve as robust alternatives to the classical likelihood ratio and other likelihood based tests. As we have noted, the weight function generated by EWD converges to 1 as its argument, the value of the density function, increases. We feel that this is the more balanced way for weighting the observations, rather than the weighing provided by the DPD, where the weights increase indefinitely with increase in the value of the density. It may also be mentioned that the proposal based on the EWD has the potential to be useful in all the situations where the DPD has been successfully applied, such as generalized linear models, survival analysis and Bayesian inference, to name a few. We hope to pursue all of these in our future research. 

In case of the hypothesis testing problem, here we have only investigated the analogues of the likelihood-ratio type tests. Other procedures, Wald-type tests based on the EWD, for example,  should also be studied. The DPD based Wald-type test has been extensively used the literature, and comparisons with EWD based tests will be interesting.

\newpage
\appendix
\section{The $B$ function of EWD($\beta$)}
\label{appA}
\begin{equation*}
\begin{aligned}
B(x) &= \frac{x^2}{\beta} \sum_{n=0}^{\infty} \frac{(-x/\beta)^n}{(n+2)! (n+1)} \\
     &= \frac{x^2}{\beta} \sum_{n=2}^{\infty} \frac{(-x/\beta)^{n-2}}{(n)! (n-1)} \\
     &= \frac{x^2}{\beta} \frac{\beta^2}{x^2} \sum_{n=2}^{\infty} \frac{(-x/\beta)^n}{(n)!} \int_{0}^{1} t^{n-2} dt \\
     &= \beta \int_{0}^{1} \frac{1}{t^2} \sum_{n=2}^{\infty} \frac{(-xt/\beta)^n}{(n)!} dt \\
     &= \beta \int_{0}^{1} \frac{\Big(\exp(-xt/\beta) - 1 + \frac{xt}{\beta} \Big)}{t^2}  dt \\
     &= \beta \int_{0}^{1} \frac{\Big(\exp(-xt/\beta) - 1 + \frac{xt}{\beta} \Big)}{t^2}  dt \\
     &= x \int_{0}^{x/\beta} \frac{\Big(\exp(-t) - 1 + t \Big)}{t^2}  dt \\
     &= x \Bigg[ - \frac{\exp(-t) - 1 + t}{t} \biggr\rvert_{0}^{x/\beta} - \int_0^{x/\beta} \Big[ \frac{\exp (-t) - 1}{t} \Big] dt \Bigg]   \\
     &= x [ I_1 - I_2]
\end{aligned}
\end{equation*}
where
\begin{equation*}
x \cdot I_1=   \beta -\beta \exp(-x/\beta) - x,
\end{equation*}
and
\begin{equation*}
\begin{aligned}
x \cdot I_2 &=x \Bigg\{\int_0^{x/\beta} \Big[ \frac{\exp (-t) - 1}{t} \Big] dt\Bigg\}\\
 &=x \Bigg\{\lim_{\Delta \rightarrow \infty}  \Bigg[\int_0^{\Delta} \Big[ \frac{\exp (-t) - 1}{t} \Big] dt - \int_{x/\beta}^{\Delta} \Big[ \frac{\exp (-t) - 1}{t} \Big] dt\Bigg] \Big\}\\
     &= x \Bigg\{\lim_{\Delta \rightarrow \infty} \Bigg[ \log(\Delta)[\exp(-\Delta) - 1]  +  \int_0^{\Delta} \Big[ \log(t) \exp (-t)  \Big] dt \\
     & \quad \quad \quad \quad - \int_{x/\beta}^{\Delta} \frac{\exp(-t)}{t} dt + \log \Big(\frac{\Delta}{x/\beta}\Big) \Bigg]\Bigg\}\\
     &= x \Bigg[ \lim_{\Delta \rightarrow \infty} \log(\Delta) \exp(-\Delta) + \int_0^\infty \log(t)\exp(-t) dt - \int_{x/\beta}^\infty \frac{\exp(-t)}{t}dt - \log(x/\beta)\Bigg]\\
&= x \cdot [0 - \gamma - \Gamma(0, x/\beta) - \log(x/\beta)]\\
     &= - x \cdot [\gamma + \Gamma(0, x/\beta)+ \log(x/\beta)].
    \end{aligned}
\end{equation*}
Here $\gamma$ is the Euler-Mascheroni constant, usually defined as
$$\gamma = \underset{n \rightarrow \infty}{\lim} \Bigg(\sum_{k=1}^{n} \frac{1}{k} - \log n\Bigg) = - \int_0^{\infty} \exp(-x) \log(x) dx,$$ and $\Gamma(\alpha,\beta)$ is the incomplete Gamma integral defined as $$\Gamma(\alpha,\beta) = \int_{\beta}^{\infty} y^{\alpha-1} \exp(-y) dy.$$ Finally, we can write
\begin{equation*}
\begin{aligned}
B(x) &= \frac{x^2}{\beta} \sum_{n=0}^{\infty} \frac{(-x/\beta)^n}{(n+2)! (n+1)} \\
    &= x [ I_1 - I_2 ]\\
    &= -x + \gamma x + \beta  - \beta \exp(-x/\beta) + x \Gamma(0, x/\beta) + x\log(x/\beta).
\end{aligned}
\end{equation*}

\section{Additional examples of simple linear regression}
\subsubsection{Hertzsprung-Russell Star Cluster data}
\label{appHR}
We consider the data for the {\em Hertzsprung-Russell diagram}
of the star cluster CYG OB1 containing 47 stars in the direction of Cygnus  \cite[Table 3, Chap. 2]{rousseeuw1987}. For these data the independent variable $x$ is the logarithm of the effective temperature at the surface of
the star ($T_e$), and the dependent variable $y$ is the logarithm of its light intensity
($L/L_0$). The data were thoroughly studied by \cite*{rousseeuw1987}
who inferred that there are two groups of data-points — four data points (in the top right corner of the scatter plot) clearly form a separate group
in comparison with the rest of the data-points. These data points are known as giants in astronomy. So, these outliers are not recording errors but are actually leverage points with the interpretation that the data are coming from two different groups. Estimates of the linear regression parameters obtained by the minimum DPD and minimum EWD methods are presented in Tables \ref{tab:t9} and \ref{tab:t10}, respectively. 
\begin{table}[h]
\centering
\caption{$\hat{{\bm{\theta}}}$ for Hertzsprung-Russell dataset using MDPDE(D($\alpha$))}
\label{tab:t9}
\begin{tabular}{lccccccc}
\toprule
Estimates & MLE & D($0.01$) & D($0.05$) & D($0.1$) & D($0.25$) & D($0.5$) & D($1$) \\ \midrule
Intercept & $6.793$ & $6.796$ & $6.803$ & $6.816$ & $-5.797$ & $-8.027$ & $-8.405$ \\
Slope & $-0.413$ & $-0.414$ & $-0.415$ & $-0.417$ & $2.440$ & $2.943$ & $3.062$ \\ \midrule
Error s.d. & $0.565$ & $0.554$ & $0.560$ & $0.566$ & $0.405$ & $0.393$ & $0.392$ \\ \bottomrule
\end{tabular}
\end{table}

\begin{table}[h]
\centering
\caption{$\hat{{\bm{\theta}}}$ for Hertzsprung-Russell dataset using MEWDE(E($\beta$))}
\label{tab:t10}
\begin{tabular}{lccccccc}
\toprule
Estimates & MLE & E($0.01$) & E($0.05$) & E($0.1$) & E($0.25$) & E($0.5$) & E($1$) \\ \midrule
Intercept  & $6.793$ & $6.795$ & $-8.236$ & $-8.395$ & $-8.537$ & $-8.408$ & $-8.373$ \\
Slope & $-0.413$ & $-0.414$ & $2.988$ & $3.024$ & $3.057$ & $3.014$ & $2.993$ \\ \midrule
Error s.d. & $0.565$ & $0.561$ & $0.355$ & $0.359$ & $0.376$ & $0.213$ & $0.123$ \\  \bottomrule
\end{tabular}
\end{table}

\begin{figure}
  \centering
  \captionsetup{justification=centering}
  \includegraphics[height=0.4\textheight]{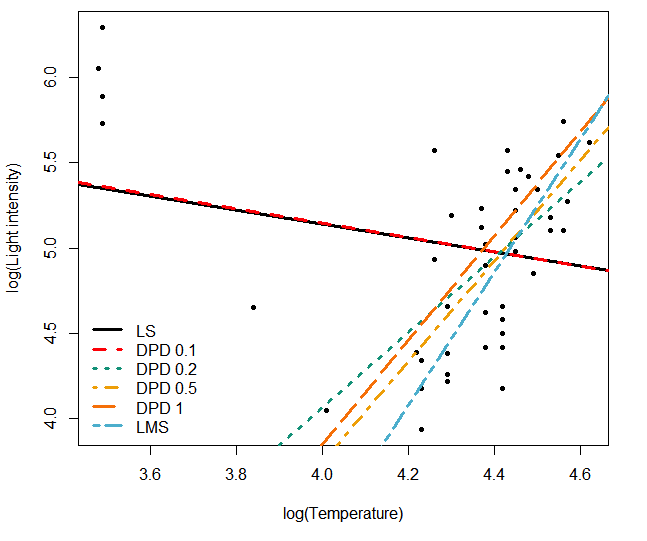}
  \caption{Data-points and fitted regression lines for Hertzsprung Russell star cluster data using least squares and minimum DPD estimates.}
\end{figure}
\begin{figure}
  \centering
  \captionsetup{justification=centering}
\includegraphics[height=0.4\textheight]{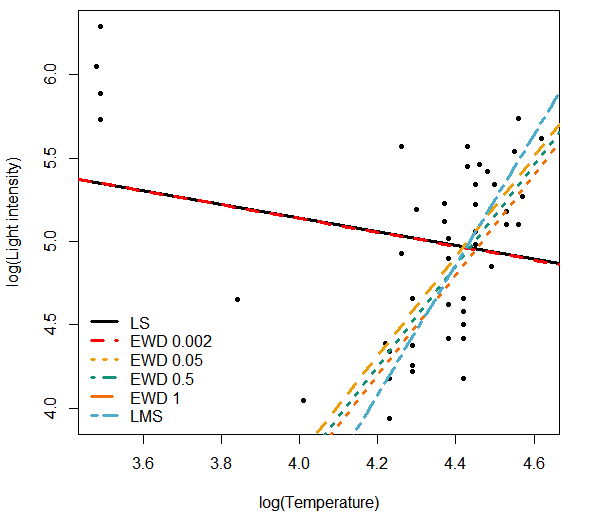} 
  \caption{Data-points and fitted regression lines for Hertzsprung Russell star cluster data using least squares and minimum EWD estimates.}
\end{figure}

We observe that
\begin{enumerate}
  \item Clearly that the estimators corresponding to $\alpha=0 \text{ and } \beta=0$ (which are  identical and also coincide with the ordinary least squares estimators) are pulled away significantly by the four leverage points and hence it is not possible to separate out the two group of data by looking at the corresponding residuals.
  \item The MDPDE with $\alpha \geq 0.25$ can successfully ignore the outliers to give excellent robust fits and are much closer to the fit generated by the LMS estimates.
  \item The MEWDE with $\beta \geq 0.05$ are strongly robust with respect to the outliers, giving excellent fits to the remaining observations.
  \item For both MDPDE and MEWDE methods, based on the residuals , we can also separate out the two group of observations – four large residuals correspond to the four giant stars.
\end{enumerate}
Thus the analysis based on DPD and EWD give stable and competitive inference in this case.

\subsubsection{Number of international telephone calls in Belgium (1950-73)}
\label{appBT}
We consider a segment of data obtained from the {\em Belgian Statistical Survey} by the Ministry of
Economy, Belgium  \cite[Table 2, Chap. 2]{rousseeuw1987}. Here, the total number (in tens of millions) of international
phone calls made in a year is the dependent variable $y$. The independent variable is the year number $x = 50, 51, \ldots, 73$. However, due to the use of another recording system (giving the total number of minutes of these calls) from the
year 1964 to 1969, the data contain heavy contamination in the y-direction in
that range. The years 1963 and 1970 are also partially affected for the same
reason. Estimates of the linear regression parameters obtained by the minimum DPD and minimum EWD methods are presented in Tables \ref{tab:t11} and \ref{tab:t12} respectively. 

\begin{table}[h]
\centering
\caption{$\hat{{\bm{\theta}}}$ for Belgium telephone data using MDPDE(D($\alpha$)).}
\label{tab:t11}
\begin{tabular}{lcccccc}
\toprule
Estimates & D($0$) & D($0.05$) & D($0.1$) & D($0.25$) & D($0.5$) & D($1$) \\ \midrule
Intercept & $-26.01$ & $-25.53$ & $-24.94$ & $-21.97$ & $-5.260$ & $-5.360$ \\
Slope &  $0.500$ & $0.500$ & $0.480$ & $0.430$ & $0.110$ & $0.110$ \\\midrule
Error s.d. & $5.380$ & $5.400$ & $5.410$ & $5.290$ & $0.110$ & $0.120$\\ \bottomrule
\end{tabular}
\end{table}

\begin{table}[h]
\centering
\caption{$\hat{{\bm{\theta}}}$ for Belgium telephone data using MEWDE(E($\beta$)).}
\label{tab:t12}
\begin{tabular}{lcccccc}
\toprule
Estimates& E($0$) & E($0.05$) & E($0.1$) & E($0.25$) & E($0.5$) & E($1$) \\ \midrule
Intercept & $-26.010$ & $-5.180$ & $-5.190$ & $-5.180$ & $-5.040$ & $-5.660$ \\
Slope &  $0.500$ & $0.110$ & $0.110$ & $0.110$ & $0.110$ & $0.120$ \\\midrule
Error s.d. & $5.380$ & $0.090$ & $0.090$ & $0.090$ & $0.080$ & $0.060$ \\ \bottomrule
\end{tabular}
\end{table}

\begin{figure}
  \centering
  \captionsetup{justification=centering}
  \includegraphics[height=0.4\textheight]{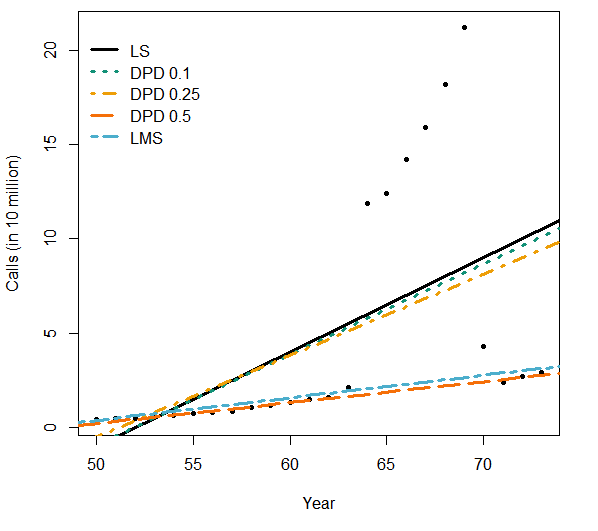}
  \caption{Plots of the data-points and fitted regression lines for Belgium telephone data using least squares and minimum DPD estimates.}
\end{figure}

\begin{figure}
  \centering
  \captionsetup{justification=centering}
\includegraphics[height=0.4\textheight]{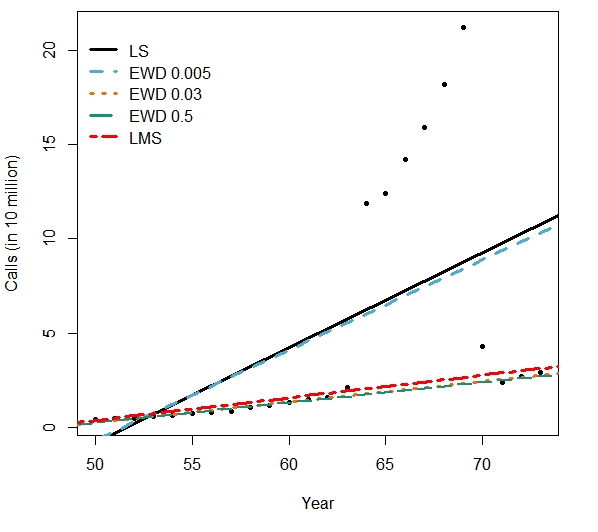}
  \caption{Plots of the data-points and fitted regression lines for Belgium telephone data using least squares and minimum EWD estimates.}
\end{figure}

We make the following observations.
\begin{enumerate}
\item It is clear that the estimators corresponding to $\alpha = 0 \text{ and } \beta = 0$ (which are identical and coincide with the ordinary LS estimators) are heavily affected by the outliers.
\item The MDPDE with $\alpha \geq 0.5$ are strongly robust with respect to the outliers, giving excellent fits to the remaining observations. While analyzing this dataset, \cite{ghosh2013} note that the slope parameter remains practically constant for all $\alpha \geq 0.4$.
\item The MEWDE with $\beta \geq 0.05$ are strongly robust with respect to the outliers, giving excellent fits to the remaining observations. We note that the estimated regression parameters do not differ by much for all $\beta \geq 0.05$ when compared to the outlier-influenced ML regression estimates.
\item The least square estimators of the regression parameters, after deleting the outlying observations corresponding to the years 1964 to 1970 are $-5.260$, $0.111$ and $0.146$, quite close to all our robust estimators.
    \end{enumerate}
Clearly, the performance of the MDPDEs and the MEWDEs are quite competitive in this example.

\subsubsection{Residual analysis of certain fits for alcohol solubility data.}
\label{app_alcohol}
In continuation with the example considered in Section \ref{alcohol} of the main article, we have, in Figure \ref{fig:appb31}, presented  the residual plots (against fitted values) of some fits (ML, LMS, ML+D (outlier deleted) and minimum EWD(0.66)) for the alcohol solubility data \citep{maronna2019}. In Figure \ref{fig:appb32} we present the kernel density estimates of the residuals of the same fits. 

As noted in Section \ref{alcohol}, we observe that the LMS and minimum EWD(0.66) procedures identify a few outliers. On the other hand, these observations remain masked in case of the maximum likelihood method, while the ML+D method does not identify any outlier. This is indicated by the lack of the long tails for the ML and ML+D methods.

\begin{figure}[htp]
  \centering
  \captionsetup{justification=centering}
  \includegraphics[width=0.8\textwidth]{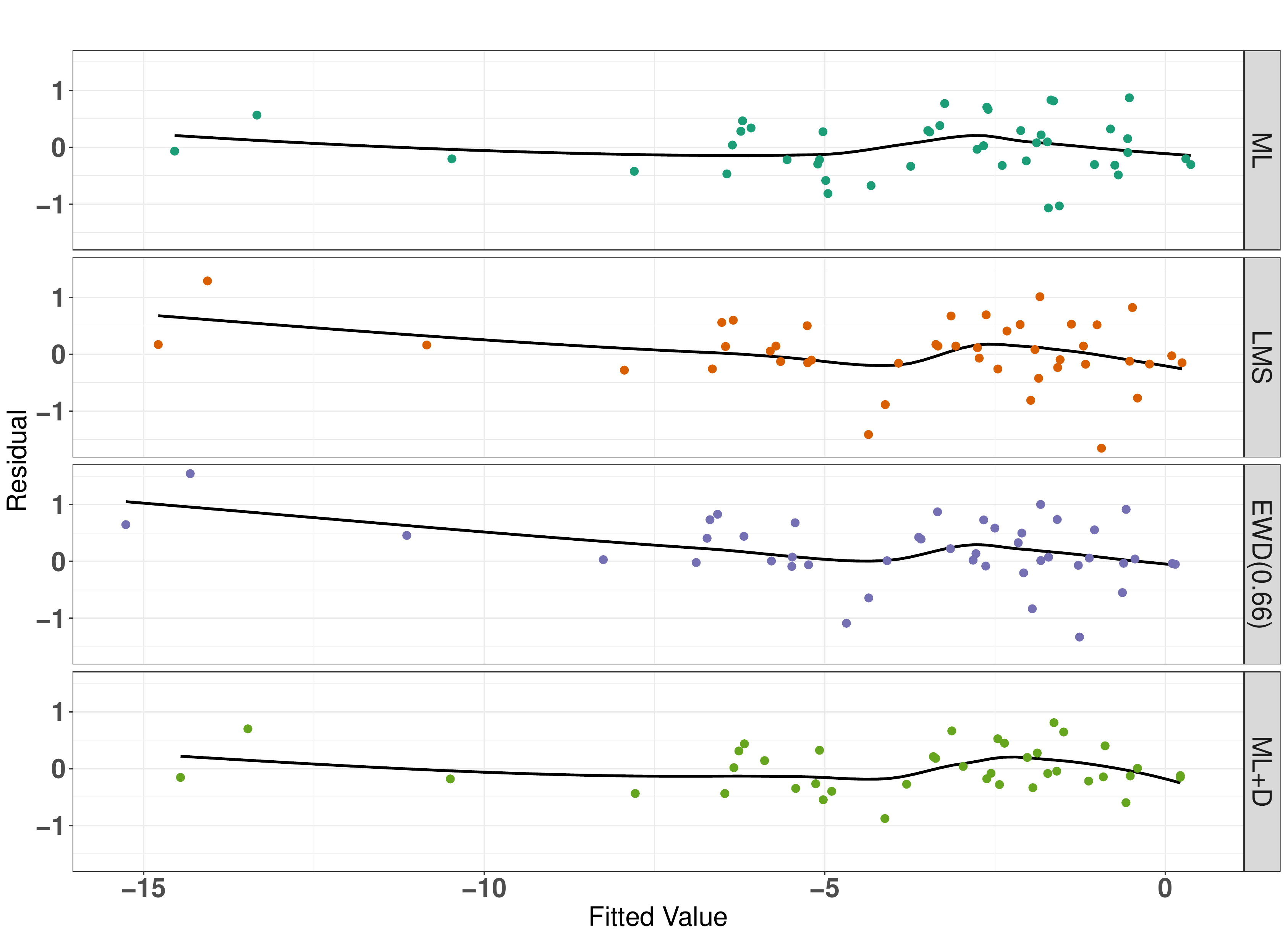}
  \caption{Scatter plots of the residuals against fitted values for ML, LMS, ML+D and minimum EWD($0.66$) fits for alcohol solubility data \citep{maronna2019}.}
  \label{fig:appb31}
  \end{figure}

\begin{figure}[htp]
  \centering
  \captionsetup{justification=centering}
  \includegraphics[width=0.8\textwidth]{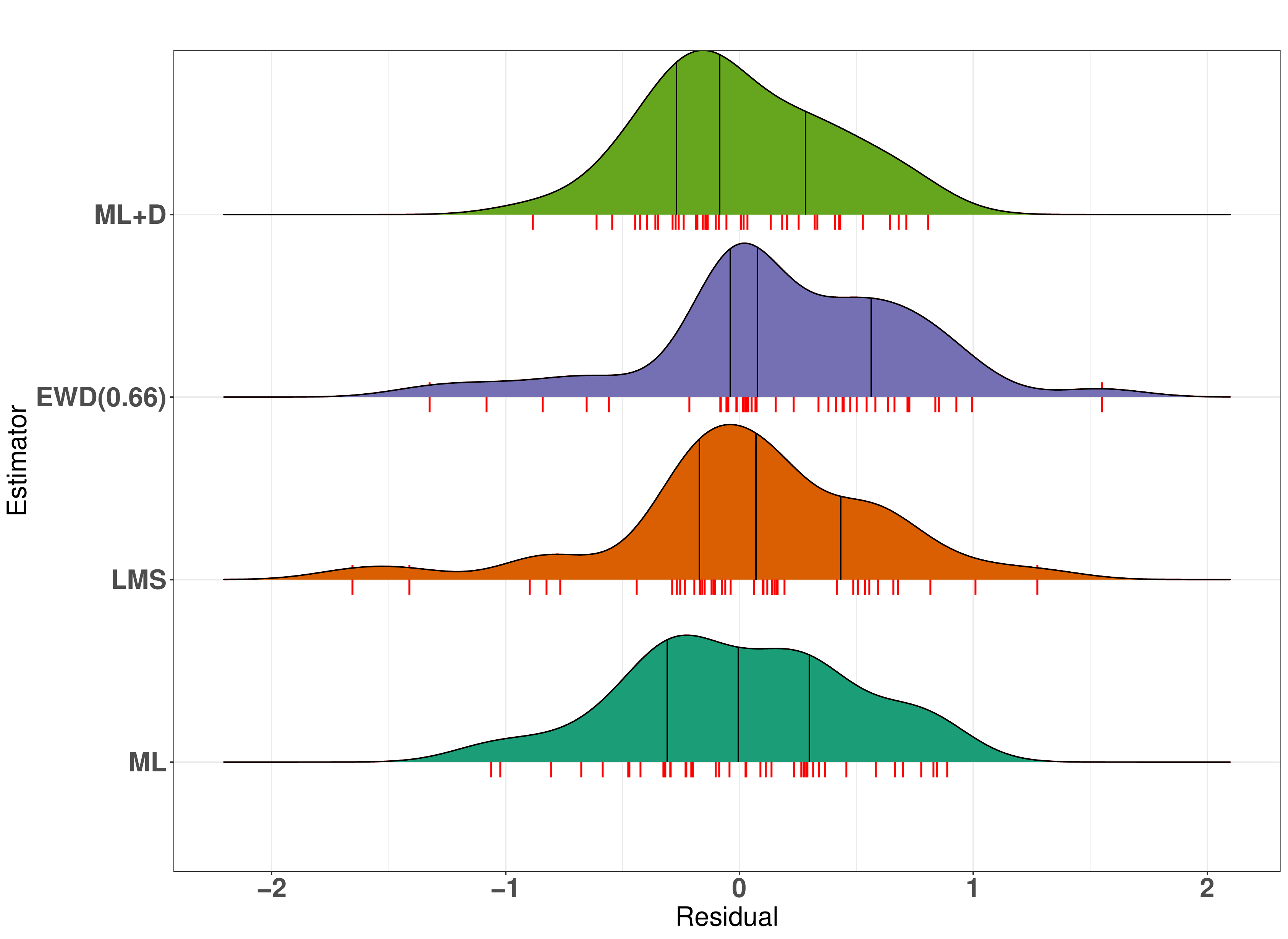}
  \caption{Kernel density estimates of the residuals arising from ML, LMS, ML+D and minimum EWD($0.66$) fits for alcohol solubility data \citep{maronna2019}. Vertical black lines correspond to $25^{\text{th}}$, $50^{\text{th}}$ and $75^{\text{th}}$ percentiles of corresponding density curves, while red rug-lines correspond to the actual residuals from which the kernel density estimates are obtained.}
  \label{fig:appb32}
  \end{figure}

\bibliographystyle{plainnat}
\bibliography{article}  

\begin{thebibliography}{26}
\providecommand{\natexlab}[1]{#1}
\providecommand{\url}[1]{\texttt{#1}}
\expandafter\ifx\csname urlstyle\endcsname\relax
  \providecommand{\doi}[1]{doi: #1}\else
  \providecommand{\doi}{doi: \begingroup \urlstyle{rm}\Url}\fi

\bibitem[Basak et~al.(2020)Basak, Basu, and Jones]{basak2020optimal}
Sancharee Basak, Ayanendranath Basu, and MC~Jones.
\newblock On the ‘optimal’ density power divergence tuning parameter.
\newblock \emph{Journal of Applied Statistics}, 2020.
\newblock URL \url{https://doi.org/10.1080/02664763.2020.1736524}.

\bibitem[Basu et~al.(1998)Basu, Harris, Hjort, and Jones]{basu1998}
Ayanendranath Basu, Ian~R Harris, Nils~L Hjort, and MC~Jones.
\newblock Robust and efficient estimation by minimising a density power
  divergence.
\newblock \emph{Biometrika}, 85\penalty0 (3):\penalty0 549--559, 1998.

\bibitem[Basu et~al.(2011)Basu, Shioya, and Park]{basu2011}
Ayanendranath Basu, Hiroyuki Shioya, and Chanseok Park.
\newblock \emph{{Statistical Inference: The Minimum Distance Approach}}.
\newblock Chapman and Hall/CRC, 2011.

\bibitem[Basu et~al.(2013)Basu, Mandal, Martin, and Pardo]{basu2013testing}
Ayanendranath Basu, Abhijit Mandal, N~Martin, and L~Pardo.
\newblock Testing statistical hypotheses based on the density power divergence.
\newblock \emph{Annals of the Institute of Statistical Mathematics},
  65\penalty0 (2):\penalty0 319--348, 2013.

\bibitem[Basu et~al.(2018)Basu, Mandal, Martin, and Pardo]{basu2018testing}
Ayanendranath Basu, Abhijit Mandal, Nirian Martin, and Leandro Pardo.
\newblock Testing composite hypothesis based on the density power divergence.
\newblock \emph{Sankhya B}, 80\penalty0 (2):\penalty0 222--262, 2018.

\bibitem[Beran(1977)]{beran1977}
Rudolf Beran.
\newblock Minimum {Hellinger} distance estimates for parametric models.
\newblock \emph{The Annals of Statistics}, 5\penalty0 (3):\penalty0 445--463,
  1977.

\bibitem[Biswas et~al.(2020)Biswas, Ghosh, and Basu]{biswas2019}
Adhidev Biswas, Abhik Ghosh, and Ayanendranath Basu.
\newblock Minimum bregman divergence and weighted likelihood: a comprehensive
  study.
\newblock \emph{Indian Statistical Institute}, 2020.

\bibitem[Bregman(1967)]{bregman1967}
Lev~M Bregman.
\newblock {The Relaxation Method of Finding the Common Point of Convex Sets and
  its Application to the Solution of Problems in Convex Programming}.
\newblock \emph{USSR Computational Mathematics and Mathematical Physics},
  7\penalty0 (3):\penalty0 200--217, 1967.

\bibitem[Broniatowski et~al.(2012)Broniatowski, Toma, and
  Vajda]{broniatowski2012decomposable}
Michel Broniatowski, Aida Toma, and Igor Vajda.
\newblock Decomposable pseudodistances and applications in statistical
  estimation.
\newblock \emph{Journal of Statistical Planning and Inference}, 142\penalty0
  (9):\penalty0 2574--2585, 2012.

\bibitem[Csisz{\'a}r(1963)]{csiszar1963}
Imre Csisz{\'a}r.
\newblock Eine informationstheoretische ungleichung und ihre anwendung auf
  beweis der ergodizitaet von markoffschen ketten.
\newblock \emph{Magyer Tud. Akad. Mat. Kutato Int. Koezl.}, 8:\penalty0
  85--108, 1963.

\bibitem[Csisz{\'a}r et~al.(1991)]{csiszar1991}
Imre Csisz{\'a}r et~al.
\newblock {Why least squares and maximum entropy? An axiomatic approach to
  inference for linear inverse problems}.
\newblock \emph{The Annals of Statistics}, 19\penalty0 (4):\penalty0
  2032--2066, 1991.

\bibitem[Ghosh and Basu(2013)]{ghosh2013}
Abhik Ghosh and Ayanendranath Basu.
\newblock Robust estimation for independent non-homogeneous observations using
  density power divergence with applications to linear regression.
\newblock \emph{Electronic Journal of Statistics}, 7:\penalty0 2420--2456,
  2013.

\bibitem[Hampel et~al.(1986)Hampel, Ronchetti, Rousseeuw, and
  Stahel]{hampel1986}
Frank~R Hampel, Elvezio~M Ronchetti, Peter~J Rousseeuw, and Werner~A Stahel.
\newblock \emph{Robust Statistics: The approach based on Influence Functions}.
\newblock John Wiley \& Sons, 1986.

\bibitem[Hettmansperger and McKean(2010)]{hettmansperger2010}
Thomas~P Hettmansperger and Joseph~W McKean.
\newblock \emph{Robust nonparametric statistical methods}.
\newblock CRC Press, 2010.

\bibitem[Huber and Ronchetti(2009)]{huber2009}
Peter~J Huber and Elvezio~M Ronchetti.
\newblock \emph{Robust Statistics}.
\newblock John Wiley \& Sons, 2009.

\bibitem[Ibragimov and Has'Minskii(1981)]{ibragimov1981}
Il'dar~Abdulovic Ibragimov and Rafail~Zalmanovich Has'Minskii.
\newblock \emph{Statistical estimation: asymptotic theory}, volume~16.
\newblock Springer Science \& Business Media, 1981.

\bibitem[Jana and Basu(2019)]{jana2019characterization}
Soham Jana and Ayanendranath Basu.
\newblock A characterization of all single-integral, non-kernel divergence
  estimators.
\newblock \emph{IEEE Transactions on Information Theory}, 65\penalty0
  (12):\penalty0 7976--7984, 2019.

\bibitem[Lindsay(1994)]{lindsay94}
B.~G. Lindsay.
\newblock Efficiency versus robustness: the case for minimum {Hellinger}
  distance and related methods.
\newblock \emph{The Annals of Statistics}, 22\penalty0 (2):\penalty0
  1081--1114, 1994.

\bibitem[Maronna et~al.(2019)Maronna, Martin, Yohai, and
  Salibi{\'a}n-Barrera]{maronna2019}
Ricardo~A Maronna, R~Douglas Martin, Victor~J Yohai, and Mat{\'\i}as
  Salibi{\'a}n-Barrera.
\newblock \emph{{Robust Statistics: Theory and Methods (with R)}}.
\newblock John Wiley \& Sons, 2019.

\bibitem[Pardo(2006)]{pardo2006statistical}
Leandro Pardo.
\newblock \emph{Statistical inference based on divergence measures}.
\newblock CRC press, 2006.

\bibitem[Roser and Ritchie(2020)]{owidhomicides}
Max Roser and Hannah Ritchie.
\newblock Homicides.
\newblock \emph{Our World in Data}, 2020.
\newblock https://ourworldindata.org/homicides.

\bibitem[Rousseeuw and Leroy(1987)]{rousseeuw1987}
Peter~J Rousseeuw and Annick~M Leroy.
\newblock \emph{Robust regression and outlier detection}, volume~1.
\newblock Wiley Online Library, 1987.

\bibitem[Simpson(1989)]{simpson1989}
Douglas~G Simpson.
\newblock Hellinger deviance tests: efficiency, breakdown points, and examples.
\newblock \emph{Journal of the American Statistical Association}, 84\penalty0
  (405):\penalty0 107--113, 1989.

\bibitem[The CIA World~Factbook()]{cia}
Central Intelligence~Agency The CIA World~Factbook.
\newblock Country comparison :: Gdp - per capita (ppp).
\newblock \emph{Central Intelligence Agency}.
\newblock URL
  \url{https://www.cia.gov/library/publications/the-world-factbook/rankorder/2004rank.html}.

\bibitem[Warwick and Jones(2005)]{warwick2005}
J~Warwick and MC~Jones.
\newblock Choosing a robustness tuning parameter.
\newblock \emph{Journal of Statistical Computation and Simulation}, 75\penalty0
  (7):\penalty0 581--588, 2005.

\bibitem[Woodruff et~al.(1984)Woodruff, Mason, Valencia, and
  Zimmering]{woodruff1984chemical}
RC~Woodruff, JM~Mason, R~Valencia, and S~Zimmering.
\newblock Chemical mutagenesis testing in drosophila: I. comparison of positive
  and negative control data for sex-linked recessive lethal mutations and
  reciprocal translocations in three laboratories.
\newblock \emph{Environmental mutagenesis}, 6\penalty0 (2):\penalty0 189--202,
  1984.

\end{thebibliography}
\end{document}